\documentclass[11pt,reqno,a4paper]{amsart}

\usepackage{amsfonts,amssymb}
\usepackage{rotating,pgf,tikz}

\usepackage{hyperref}
\usepackage{graphicx,color}
\usepackage[normalem]{ulem}
\usepackage[hmargin=3.0cm,bmargin=3.0cm,tmargin=3.5cm]{geometry}

\title[Extremal domains and P\'{o}lya-type inequalities]{Extremal domains and P\'{o}lya-type inequalities for the 
Robin Laplacian on rectangles and unions of rectangles}

\author{Pedro Freitas}
\author{James Kennedy}

\address{Departamento de Matem\'atica, Instituto Superior T\'ecnico, Universidade de Lisboa, Av. Rovisco Pais 1,
P-1049-001 Lisboa, Portugal {\rm and}
Grupo de F\'isica M\'atematica, Faculdade de Ci\^encias, Universidade de Lisboa,
Campo Grande, Edif\'icio C6, P-1749-016 Lisboa, Portugal}
\email{psfreitas@fc.ul.pt}
\address{Grupo de F\'isica M\'atematica, Faculdade de Ci\^encias, Universidade de Lisboa,
Campo Grande, Edif\'icio C6, P-1749-016 Lisboa, Portugal}
\email{jbkennedy@fc.ul.pt}

\newtheorem{theorem}{Theorem}[section]
\newtheorem{lemma}[theorem]{Lemma}
\newtheorem{proposition}[theorem]{Proposition}
\newtheorem{corollary}[theorem]{Corollary}

\newtheorem{conjecture}[theorem]{Conjecture}
\newtheorem*{conjecture*}{Conjecture}

\newtheorem{thmx}{Theorem}
\newtheorem{corx}[thmx]{Corollary}

\theoremstyle{remark}
\newtheorem{remark}[theorem]{Remark}
\newtheorem{definition}[theorem]{Definition}

\numberwithin{equation}{section}
\numberwithin{figure}{section}

\newcommand{\R}{\mathbb{R}}
\newcommand{\N}{\mathbb{N}}

\newcommand\soutb{\bgroup\markoverwith{\textcolor{blue}{\rule[.5ex]{2pt}{1pt}}}\ULon}
\newcommand\soutr{\bgroup\markoverwith{\textcolor{red}{\rule[.5ex]{2pt}{1pt}}}\ULon}

\newcommand{\bo}{{\rm O}}
\newcommand{\so}{{\rm o}}
\newcommand{\ds}{\displaystyle}

\newcommand{\dint}{\ds\int}
\newcommand{\eqskip}{ \vspace*{2mm}\\ }
\newcommand{\fr}[2]{\frac{\ds #1}{\ds #2}}
\newcommand{\keig}[3]{{\lambda_{#1}}(#2,#3)} 
\newcommand{\keigd}[2]{{\lambda_{#1}}(#2,\infty)} 
\newcommand{\keign}[2]{{\lambda_{#1}}(#2,0)} 
\newcommand{\eigcount}[3]{N_{#1,#2}(#3)} 
\newcommand{\unionsquare}[1]{\mathcal{U}_{#1}} 
\newcommand{\rect}[2]{\mathcal{R}_{#2}(#1)} 
\newcommand{\sq}[1]{\mathcal{S}_{#1}} 
\newcommand{\interv}[1]{\mathcal{I}_{#1}} 
\newcommand{\keigopt}[3]{\lambda_{#1}^{\rm +} (#2,#3)} 
\newcommand{\keigoptrect}[3]{\lambda_{#1}^{\rm *} (#2,#3)} 
\newcommand{\eigmode}[4]{\lambda_{(#1,#2)}(#3,#4)} 
\newcommand{\eigmodeoptrect}[4]{\lambda_{(#1,#2)}^{\rm rect} (#3,#4)} 

\newcommand{\sumeigopt}[3]{\sigma_{#1}^{\rm +} (#2,#3)} 
\newcommand{\unioncube}[1]{\mathcal{Q}_{#1}} 
\newcommand{\unionball}[1]{\mathcal{B}_{#1}} 

\newcommand{\family}{\mathcal{A}} 
\newcommand{\secfamily}{\mathcal{B}} 

\begin{document}

\begin{abstract}
We show that eigenvalues of the Robin Laplacian with a positive boundary parameter $\alpha$ on rectangles and unions of rectangtes satisfy
P\'{o}lya-type inequalities, albeit with an exponent smaller than that of the corresponding Weyl asympotics for a fixed domain.
We determine the optimal exponents in either case, showing that they are different in the two situations. Our approach to proving these
results includes a characterisation of the corresponding extremal domains for the $k^{\rm th}$ eigenvalue in regions of the $(k,\alpha)-$plane.
\end{abstract}



\thanks{The work of the authors was supported by the Funda{\c{c}}{\~a}o para a Ci{\^e}ncia e a Tecnologia, Portugal, via the program ``Investigador FCT'', reference
IF/01461/2015 (JK), and project PTDC/MAT-CAL/4334/2014 (PF and JK)}


\maketitle

\section{Introduction}
\label{sec:intro}

Given a planar domain $\Omega$ with a sufficiently smooth boundary $\partial\Omega$, consider the equation
\begin{equation}\label{eigeq}
 \Delta u + \tau u = 0 \mbox{ in } \Omega
\end{equation}
with one of the following boundary conditions
\[
 \begin{array}{lll}
  u = 0, & x \in \partial\Omega & \mbox{(Dirichlet)}\eqskip
  \fr{\partial u}{\partial \nu} = 0, &  x \in \partial\Omega & \mbox{(Neumann})
 \end{array},
\]
where $\nu$ is the outer unit normal defined on $\partial\Omega$. Denoting by $\gamma_{k}$ and $\mu_{k}$ the Dirichlet and
Neumann eigenvalues, respectively, corresponding to the numbers $\tau$ for which nontrivial solutions $u$ of the above equation exist, we have
\[
0 < \gamma_{1}\leq \gamma_{2} \leq \dots
\]
and
\[
 0 = \mu_{1}\leq \mu_{2} \leq \dots,
\]
with both sequences being unbounded.

In the second volume of his book {\it Mathematics and plausible reasoning}, P\'{o}lya conjectured that
\begin{equation}\label{originpolyaconj}
  \mu_{k} < \fr{4k\pi}{A} < \gamma_{k}, \; k=1,2,\dots
\end{equation}
for planar domains with area $A$~\cite[pp.~51--53]{poly1}. After stating this conjecture, P\'{o}lya went on to say that these inequalities are satisfied in the
case of rectangles and that it had been this particular case that had suggested the conjecture.

A few years later, P\'{o}lya himself provided a remarkably simple and elegant argument to prove~\eqref{originpolyaconj} in the Dirichlet case for
plane-covering (tiling) domains, that is, domains which ``cover the whole plane without gaps and without overlapping''~\cite{poly2}. In the same article, P\'{o}lya
also provided a shaper version of the Neumann part of the conjecture by replacing $\mu_{k}$ with $\mu_{k+1}$, which he then proved
for a smaller class of domains, with the general result for tiling domains being obtained not long afterwards by Kellner \cite{kell}.

In~\cite{poly1} P\'{o}lya gave some further heuristic arguments as to why conjecture~\eqref{originpolyaconj} should be true in general, such as the fact
that it holds for the first two eigenvalues of general domains, which follows from the Faber--Krahn and Hong--Krahn--Szego inequalities, and the Szeg\H{o}--Weinberger
inequality in the Dirichlet and Neumann cases, respectively. He also mentioned that, upon division by $k$, all three terms in~\eqref{originpolyaconj} have the
same limit, as a consequence of the Weyl asymptotics. However, he left out one of the most compelling pieces of evidence for~\eqref{originpolyaconj} to
hold, probably because this was itself a conjecture at the time, namely the two-term Weyl asymptotics
\[
 \tau_{k} = \fr{4k\pi}{A} \pm 2\sqrt{k \pi} \fr{L}{A^{3/2}} + \so\left(k^{1/2}\right)
\]
where $L$ denotes the perimeter of $\Omega$ and the $+$ and $-$ signs correspond to Dirichlet and Neumann boundary conditions, respectively~\cite{sava}.
In fact, not only does the second term in the above asymptotics support the conjecture, but it also shows that the latter is asymptotically correct for any
particular domain.

Although the conjecture remains open to this day, progress has been made with the best results so far for general planar domains being
\begin{equation}
\label{eq:polya-best}
 \mu_{k} \leq \fr{8\pi}{A}(k-1) \mbox{ and } \fr{2\pi}{A}k \leq \gamma_{k}, \;\; k=1,2,\dots.
\end{equation}
Here the Dirichlet bound was proved by Li and Yau in 1983~\cite{liyau}, and as was later realised, could also be recovered from work by Berezin~\cite{bere},
while the Neumann bound is due to Kr\"{o}ger in 1992~\cite{krog}.

In this paper we want to study the same problem in the Robin case, that is, we consider equation~\eqref{eigeq} together
with the boundary condition
\begin{equation}\label{robinboundary}
\begin{array}{lll}
  \fr{\partial u}{\partial \nu} + \alpha u= 0, & x \in \partial\Omega &\mbox{(Robin)},
\end{array}
\end{equation}
where $\alpha$ is a positive real parameter. It is a natural question to ask what form, if any, should an inequality of the same type
as~\eqref{originpolyaconj} take for the eigenvalues of the above problem. We first note that the corresponding eigenvalues, which we shall denote
by $\keig{k}{\Omega}{\alpha}$, also satisfy an inequality of Faber--Krahn type, namely,
\[
 \keig{1}{B}{\alpha} \leq \keig{1}{\Omega}{\alpha}
\]
for all positive $\alpha$, where $B$ denotes the ball with the same measure as $\Omega$ -- for planar domains this was proved by Bossel~\cite{boss}
and generalised by Daners~\cite{dane} to higher dimensions. The corresponding Hong--Krahn--Szego inequality was proved by the second author in~\cite{kenn}.
It might thus seem reasonable to expect the sequence of $\keig{k}{\Omega}{\alpha}$ to have a behaviour analogous to that of the Dirichlet problem.
However, it is known that the two-term Weyl asymptotic is in fact the same as that of the Neumann problem~\cite{frge}, namely,
\begin{equation}\label{robinweyl}
 \keig{k}{\Omega}{\alpha} =  \fr{4k\pi}{A}- 2\sqrt{k \pi} \fr{L}{A^{3/2}} + \so\left(k^{1/2}\right).
\end{equation}
A consequence of this is that clearly there cannot exist a lower bound for $\keig{k}{\Omega}{\alpha}$ with the same power of $k$ as in the Dirichlet case
which is also compatible with the first term of~\eqref{robinweyl}. To some extent, this was already pointed
out in~\cite{anfrke}, where it was seen, by considering the sequence of domains $\unionball{k}$ consisting of the $k$ disjoint unions of equal balls, that the asymptotic
behaviour of the infimum of the $k^{\rm th}$ Robin eigenvalue among domains of equal volume will satisfy
\[
 \inf_{|\Omega|=A} \keig{k}{\Omega}{\alpha} \leq \keig{k}{\unionball{k}}{\alpha}\leq 2\alpha \left(\fr{k \pi}{A}\right)^{1/2}.
\]
Thus, although it might still possible to consider inequalities of the type $\keig{k}{\Omega}{\alpha}\geq c k$, the constant $c$ would have to
depend on $\Omega$ in a nontrivial fashion and cannot, in any case, be optimal in an asymptotic sense as $k$ approaches infinity.

In order to gain some insight into this issue, it is, of course, tempting to follow P\'{o}lya's approach and see what happens in the simpler case of rectangles
or possibly even tiling domains. However, for Robin boundary conditions  there is no explicit closed form for the eigenvalues in the former case,
while in the latter two of the key ingredients used in~\cite{poly2}, namely the simple rescaling formula and monotonicity by inclusion which are fundamental in the Dirichlet proof do not apply for~\eqref{robinboundary}.

A first purpose of this paper is thus to obtain further understanding of this problem by studying the existence of P\'{o}lya-type
inequalities of the form
\[
 c k^\beta \leq \keig{k}{\Omega}{\alpha},
\]
where the constant $c$ depends only on the boundary parameter $\alpha$ and is independent of $\Omega$ within 
families of domains with a given area. A key point is the determination of the optimal power $\beta$. The two classes of domains which we shall consider here are
rectangles and disjoint unions of rectangles. There are two reasons for studying these two families. On the one hand, it is natural to
try to understand what happens in the case of
rectangles, by analogy with the Dirichlet case. On the other hand, and as will become clear, the two problems yield different values of $\beta$ and thus illustrate the
essential differences that may be expected even within the Robin problem. Furthermore, the behaviour for unions of rectangles should, in principle, be closer to what is to
be expected to happen in the general problem.

In this direction, our main result may be summarised as follows.
\begin{thmx}[P\'{o}lya-type inequalities]
\label{thm:polya}
Given positive numbers $\alpha$ and $A$, there exist positive constants $c_{r}$ and $c_{u}$, depending only on $\alpha$ and $A$,
such that the Robin eigenvalues satisfy
\[
 c_{r} k^{2/3} \leq \keig{k}{\Omega}{\alpha},
\]
for all rectangles with given area $A$, and
\[
 c_{u} k^{1/2} \leq \keig{k}{\Omega}{\alpha},
\]
for all unions of rectangles with total area $A$. Furthermore, the exponents $2/3$ and $1/2$ are optimal.
\end{thmx}

This theorem follows directly from the more detailed Theorems~\ref{thm:k-squares} and~\ref{thm:rectangles} below; in particular,
they, together with the fact that the eigenvalues are increasing with $\alpha$, allow us to give explicit lower bounds on the constants $c_{r}$ and $c_{u}$. 

The issue of determining the optimal constant in P\'{o}lya's conjecture for the Dirichlet and Neumann problems is naturally related to that
of considering the extremal values of the eigenvalues $\gamma_{k}$ and $\mu_{k}$. In fact, in the case of general domains with a measure
restriction, this connection is much stronger than had been previously thought, in that it was shown recently that P\'{o}lya's
conjecture is equivalent to the first term in the asymptotic behaviour of the extremal values being the same as that in the Weyl asymptotics
for a fixed domain~\cite{colels}. This effect is a direct consequence of the subadditivity and superadditivity of the sequences of (dimensionally normalised)
extremal eigenvalues in the Dirichlet and Neumann cases, respectively.

Taking this into consideration, the approach we follow in this paper is mixed, in the sense that we will prove Theorem~\ref{thm:polya} by
studying the sequence of extremal sets in both cases. We recall that even for Dirichlet eigenvalues, which in the case of rectangles
are known explicitly, it is a nontrivial problem to show that the sequence of extremal rectangles does converge to the square as $k$ goes to
infinity~\cite{anfr2}. This result, which is closely related to a lattice point counting problem, has also been extended to Neumann boundary
conditions~\cite{vdbbugi}, higher dimensions~\cite{gila},
and several variants with a more geometric~\cite{ar} or number-theoretic flavour~\cite{guwa,arla,lali1,lali2,ma,mast}.
This is thus also a motivation to study the evolution of the sequence of extremal rectangles. However, since for the Robin problem the asymptotic
extremal domain is no longer thought to be the square, we also consider the situation where we allow for arbitrary unions of rectangles. This
is a natural setting to consider, due to the considerations made above and, in particular, the results obtained in~\cite{anfrke}, where the following
conjecture was made~(\cite[Section~5]{anfrke}; see also \cite[Open Problem~4.38]{bufrke}).
\begin{conjecture*}[Optimality of $k$ equal balls]
\label{conj:robin-optimisers}
Fix a dimension $d\geq 2$ and $k\geq 3$. Then there exists some $\alpha_k^\ast>0$ depending only on $k$ and $d$ such that
\begin{displaymath}
	 \keig{k}{\unionball{k}}{\alpha} \leq \keig{k}{\Omega}{\alpha}
\end{displaymath}
for all $\alpha \in (0,\alpha_k^\ast]$ and all (sufficiently smooth) domains $\Omega\subset \R^d$ with $|\Omega|=1$, where $\unionball{k}$ is the disjoint union of $k$ equal balls of total volume $1$. Moreover, $\unionball{k}$ is \emph{not} optimal for $\alpha > \alpha_k^\ast$, and $\alpha_k^\ast \to \infty$ as $k\to \infty$.
\end{conjecture*}

Here we shall provide strong supporting evidence for this conjecture by essentially proving it in the restricted setting of rectangles and unions of rectangles (with $k$ equal squares taking the role of $k$ equal balls); we also expect some of the tools and insights we develop to be of use when investigating the conjecture on more general domains.

For any positive values of the area $A$ and boundary parameter $\alpha$, and any positive integer $k$, we will write $\keigopt{k}{A}{\alpha}$ to stand for the extremal quantity
\begin{displaymath}
	\inf \{ \keig{k}{\Omega}{\alpha}:\Omega \subset \R^2 \text{ is a disjoint union of rectangles, } |\Omega|=A \},
\end{displaymath}
and we let $\unionsquare{k}$ denote the disjoint union of $k$ equal squares of the same total area $A$. Our main result in this context is then
\begin{thmx}[Optimality of $k$ equal squares]
\label{thm:k-squares}
There exists an absolute positive constant $C_1$ such that, for any finite disjoint union of rectangles $\Omega$ having total area $A$, we have
\begin{displaymath}
	\keig{k}{\unionsquare{k}}{\alpha} < \keig{k}{\Omega}{\alpha}
\end{displaymath}
whenever $\alpha \leq C_1 k^{1/2}A^{-1/2}$, where $\unionsquare{k}$ has total area $A$. Furthermore, for such
pairs $\alpha, k$, we have
\begin{equation}
\label{eq:opt-value-est}
	\frac{4\pi^2 k\alpha}{A^{1/2}\pi^2 k^{1/2} + 2A\alpha} < \keigopt{k}{A}{\alpha} \leq \frac{4k^{1/2}\alpha}{A^{1/2}}.
\end{equation}
\end{thmx}
\begin{remark} 
 It is possible to improve the upper bound for $\keigopt{k}{A}{\alpha}$ given above, by using the more precise (and complicated)
 bounds given in the Appendix -- see Proposition~\ref{prop:boundkequalsquares}.
\end{remark}

Our proof gives an explicit estimate on the constant, namely $C_1>\fr{\pi^2}{18}\left(7 - 2\sqrt{10}\right)\approx 0.370$ -- see Sections~\ref{sec:uk-in-eu-fat} and~\ref{sec:disjoint-unions} for the derivation of the constant for rectangles and unions of rectangles, respectively. Thus, for any fixed positive $\alpha$, for $k$ sufficiently large (or equivalently, for any fixed $k$ for $\alpha$ sufficiently small, for an explicitly given value), the minimiser of $\keig{k}{\,\cdot\,}{\alpha}$ among all unions of rectangles of fixed total area is the domain consisting of the disjoint union of $k$ equal squares. 

This also allows us to obtain estimates on the constants $c_{u}$ and $c_{r}$ appearing in Theorem~\ref{thm:polya}. We thus have, for instance,
\[
 c_{u} \geq \left\{
 \begin{array}{ll}
  \fr{4\pi^2 \alpha}{A^{1/2}\left(\pi^2 + 2 \alpha A^{1/2}\right)}, & \alpha < C_{1} k^{1/2}A^{-1/2}\eqskip
  \fr{4\pi^2 C_{1}}{A\left(\pi^2+2A^{1/2}C_{1}\right)}, & \alpha \geq C_{1} k^{1/2}A^{-1/2}
 \end{array}
\right.
\]
with $C_{1}$ as above.

Moreover, the form of the given relationship between $\alpha$ and $k$ is optimal:

\begin{thmx}[Loss of optimality of $k$ equal squares]
\label{thm:brexit}
There exists a further absolute positive constant $C_2$ such that $\keig{k}{\,\cdot\,}{\alpha}$ is not minimised among all finite unions of rectangles of fixed total area $A$ by $\unionsquare{k}$ whenever $\alpha \geq C_2 k^{1/2} A^{-1/2}$.
\end{thmx}

In fact, we obtain the exact form of the curve where having three of the squares of $\unionsquare{k}$ replaced by one larger square will provide the same eigenvalue; based on results from \cite{anfrke} for balls and numerics it is to be expected, although it is not yet known, that this is the exact point where $\unionsquare{k}$ stops being the optimiser. Our result again includes an explicit estimate on $C_2$; see Theorem~\ref{transition3to1} for the details. In fact, for fixed $k$, as $\alpha \to \infty$ the optimiser converges to its Dirichlet counterpart, as we show in Theorem~\ref{thm:robin-dirichlet-convergence}. Moreover, for any \emph{fixed} domain $\Omega$ there exists a
constant $C_\Omega > 0$, which numerically generally appears to be close to the numerically optimal $C_2$, such that $\keig{k}{\Omega}{\alpha} < \keig{k}{\unionsquare{k}}
{\alpha}$ whenever $\alpha \geq C_\Omega k^{1/2}A^{-1/2}$, as we show in Section~\ref{sec:higher-dimension}. All this highlights a sharp difference in qualitative behaviour
between regions of the form $\alpha \leq c k^{1/2}A^{-1/2}$ and $\alpha \geq c k^{1/2} A^{-1/2}$.

Nevertheless, the fact that, for any fixed $\alpha$, $\unionsquare{k}$ becomes the extremal domain for all sufficiently large $k$, allows us to describe the asymptotic behaviour of the optimal values $\keigopt{k}{A}{\alpha}$ as $k\to\infty$ for fixed $A$ and $\alpha$.

\begin{corx}[Asymptotic behaviour of optimal unions of rectangles]
\label{cor:optimal-asymptotic}
For any given positive values of the area $A$ and boundary parameter $\alpha$,
\begin{displaymath}
	\lim_{k\to\infty} \frac{\keigopt{k}{A}{\alpha}}{k^{1/2}} = \frac{4\alpha}{A^{1/2}},
\end{displaymath}
and indeed, $\keigopt{k}{A}{\alpha}$ has the same asymptotic behaviour as $k$ goes to infinity as $\unionsquare{k}$, namely,
\begin{equation}
\label{eq:optimal-asymptotic}
\begin{array}{lll}
	\keigopt{k}{A}{\alpha} & = & \fr{4\alpha}{A^{1/2}} k^{1/2} - \fr{2\alpha^2}{3} + \fr{4A^{1/2}\alpha^3}{45} k^{-1/2}\eqskip
	& & \hspace*{5mm}- \fr{8A\alpha^4}{945}k^{-1} + \fr{4A^{3/2}\alpha^5}{1475} k^{-3/2} + \bo(k^{-2})
\end{array}
\end{equation}
as $k \to \infty$.
\end{corx}
\begin{proof} The limit follows from the bounds in Theorem~\ref{thm:k-squares}, while the asymptotic expansion may be obtained from that of the first
 Robin eigenvalue given by~\eqref{eq:eig1expa} by taking $a=(A/k)^{1/2}$ and noting that $a\to 0$ as $k\to\infty$.
\end{proof}

As a further consequence, we can obtain two-sided estimates on the smallest possible value of the sum of the first $k \geq 1$ eigenvalues,
\begin{equation}
\label{eq:sumeigopt}
	\sumeigopt{k}{A}{\alpha} := \inf_\Omega \sum_{j=1}^k \keig{j}{\Omega}{\alpha},
\end{equation}
where the infimum is taken over all disjoint unions $\Omega$ of rectangles such that $|\Omega|=A$, for fixed $A>0$ and $\alpha > 0$. Li and Yau \cite{liyau} famously obtained a sharp lower bound on the corresponding sum in the Dirichlet case, namely $2\pi k^2/A$ (in two dimensions); the lower bound on $\gamma_k$ from \eqref{eq:polya-best} is obtained as a direct consequence of it. Here, $\sumeigopt{k}{A}{\alpha}$ must behave asymptotically like $k^{3/2}$, not $k^2$.
\begin{corx}[Asymptotic behaviour of optimal sums of eigenvalues]
\label{cor:sumeigopt}
Fix $A>0$ and $\alpha>0$. Then
\begin{displaymath}
	\frac{8}{3}A^{-1/2}\alpha \leq \liminf_{k\to\infty} \frac{\sumeigopt{k}{A}{\alpha}}{k^{3/2}} \leq
	\limsup_{k\to\infty} \frac{\sumeigopt{k}{A}{\alpha}}{k^{3/2}} \leq 4A^{-1/2}\alpha.
\end{displaymath}
\end{corx}
\begin{remark}
The method of proof of this result, which we give in Section~\ref{sec:sums}, can be easily extended to provide explicit bounds for $\sumeigopt{k}{A}{\alpha}$ for any given $k$.
\end{remark}

As we saw in Theorem~\ref{thm:polya}, even if we restrict our attention just to \emph{rectangles} rather than disjoint unions, then we
still have that the growth
of the extremal eigenvalues in $k$ lies below the corresponding Weyl asymptotics. Moreover, now the corresponding sequence of optimisers is unbounded with
the ratio of side lengths tending to infinity with $k$. The following theorem also gives the rate at which this happens. Here, we set
\begin{displaymath}
	\keigoptrect{k}{A}{\alpha} := \inf\{\keig{k}{\Omega}{\alpha}: \Omega \subset \R^2 \text{ is a rectangle, } |\Omega|=A\},
\end{displaymath}
and denote by $A^{1/2}a_k^\ast$ the longer side of the rectangle yielding the optimum, for fixed $\alpha$.

\begin{thmx}[Optimality and asymptotic behaviour of rectangles]
\label{thm:rectangles}
For any fixed $A>0$, $\alpha>0$ and $k\geq 2$, there exists a positive constant $C_3$ such that
\begin{displaymath}
	\fr{3\pi^2\alpha^{2/3}}{A^{2/3}\left(\pi^2+2A^{1/2}\alpha\right)^{2/3}} (k-2)^{2/3} \leq \keigoptrect{k}{A}{\alpha} 
	\leq \fr{3\pi^{2/3}\alpha^{2/3}}{A^{2/3}} k^{2/3},
\end{displaymath}
the lower bound holding whenever $\alpha \leq C_3 k^{1/2}$. In particular, there exists a positive constant $C_4=C_4(\alpha,A)$ such that
$\keigoptrect{k}{A}{\alpha}\geq C_4k^{2/3}$ for all $k\geq 3$. In addition,
\begin{displaymath}
	\lim_{k\to\infty} \frac{\keigoptrect{k}{A}{\alpha}}{k^{2/3}} = 3\left(\fr{\pi\alpha}{A}\right)^{2/3}.
\end{displaymath}
Moreover, there exist constants $c_1,c_2$ depending on $\alpha$ such that $c_1 k^{2/3} \leq a_k^\ast \leq c_2 k^{2/3}$ for all $k\geq 1$.
\end{thmx}

An explicit estimate for the constant $C_3$ is given in \eqref{eq:rectangles-rectangles-comparison}; asymptotically valid (for large $k$) bounds on $c_1$ and $c_2$ are given in \eqref{eq:c1} and \eqref{eq:a-upper-bound}, respectively.

All of the above is in sharp contrast with what happens in the Dirichlet and Neumann cases, where the extremal sets converge to a fixed domain as $k$ becomes large~\cite{anfr2,anfr1}.

In the next section we introduce notation and recall some basic properties about Robin eigenvalues. We then establish the existence
of minimisers and prove that the square minimises the first eigenvalue among unions of rectangles with fixed area in Sections~\ref{sec:rectangles}
and~\ref{sec:isoperimetric}, respectively. As far as we are aware, the latter result has not previously appeared in the literature and due to lack of explicit
solutions is slightly more complicated to prove that its Dirichlet or Neumann counterparts. Before moving on to the proofs of Theorems~\ref{thm:polya}--\ref{thm:rectangles},
we show that, for each fixed $k\geq 1$, the Robin minimisers do in fact converge to Dirichlet minimisers as $\alpha$ goes to infinity, thus behaving in a way similar to the eigenvalues themselves
in this respect (Section~\ref{sec:convergence-to-dirichlet}). The proofs of the main results are then given in Section~\ref{sec:opt-rect} (rectangles, Theorem~\ref{thm:rectangles})
and Section~\ref{sec:uk-in-eu} (unions of rectangles, Theorems~\ref{thm:k-squares} and~\ref{thm:brexit}). 
In Section~\ref{sec:sums}, we give the proof of Corollary~\ref{cor:sumeigopt} as well as some further remarks on the problem of minimising the sum of the first $k$ Robin eigenvalues. In Section~\ref{sec:higher-dimension}, we briefly discuss higher-dimensional versions of these results: the principles should be the same, and we indicate how the exponents of interest should depend on the dimension.
Finally, in the Appendix we collect several sharp estimates for the eigenvalues of the Robin
problem on a bounded interval which are used throughout the text and which we believe to be useful in their own right.
We draw particular attention to the asymptotic behaviour of the first and second Robin eigenvalues of an interval as its length $a$ tends to zero,
which are of orders $a^{-1}$ and $a^{-2}$, respectively; it is this differentiating behaviour that will drive many of our results.

\section{Notation and basic properties of the Robin Laplacian}
\label{sec:notation}

Let $\Omega \subset \R^d$, $d\geq 1$, be a bounded, not necessarily connected, domain with Lipschitz boundary $\partial\Omega$.
For given $\alpha>0$, we will be interested in the eigenvalues
\begin{displaymath}
	0 < \keig{1}{\Omega}{\alpha} \leq \keig{2}{\Omega}{\alpha} \leq \ldots \to \infty
\end{displaymath}
of the Robin Laplacian, namely the operator on $L^2(\Omega)$ formally associated with the sesquilinear form $q_\alpha: H^1(\Omega) \times H^1(\Omega) \to \mathbb{C}$ given by
\begin{displaymath}
	q_\alpha (u,v) = \int_\Omega \nabla u\cdot \overline{\nabla v}\,\textrm{d}x+\alpha\int_{\partial\Omega} u\overline{v}\,\textrm{d}\sigma,
\end{displaymath}
see, e.g., \cite[Section~4.2]{bufrke} for more details. The Neumann Laplacian corresponds to $\alpha=0$, while the Dirichlet Laplacian is formally obtained for $\alpha = \infty$; we will thus also write
\begin{displaymath}
	\keign{k}{\Omega} = \mu_k (\Omega) \quad \text{and} \quad \keigd{k}{\Omega} = \gamma_k (\Omega)
\end{displaymath}
for the Neumann and Dirichlet eigenvalues, respectively, where $\keign{1}{\Omega}=0$. We also recall the following standard continuity result with respect to $\alpha$
again, see, e.g., \cite[Section~4.2]{bufrke}).

\begin{lemma}
\label{lem:parameter-continuity}
Let $\Omega \subset \R^d$, $d\geq 1$, be Lipschitz. Then for each $k\geq 1$, the mapping
\begin{displaymath}
	\alpha \mapsto \keig{k}{\Omega}{\alpha}
\end{displaymath}
is a continuous and monotonically increasing function of $\alpha \in [0,\infty]$. In particular, $\keig{k}{\Omega}{\alpha} \to \keign{k}{\Omega} = \mu_k (\Omega)$ from above as $\alpha \to 0$, and $\keig{k}{\Omega}{\alpha} \to \keigd{k}{\Omega} = \gamma_k (\Omega)$ from below as $\alpha \to \infty$.
When $k=1$ and $\Omega$ is connected, this function is analytic with strictly negative second derivative everywhere.
\end{lemma}

The concrete choices of $\Omega$ which will be most relevant for us in the sequel will be denoted as follows:
\begin{center}
\begin{tabular}{cl}
	$\interv{a}$ & any interval of length $a$;\\
	$\rect{a}{A}$ & any rectangle of area $A$ and side lengths $A^{1/2}a$ and $A^{1/2}/a$;\\
	$\sq{a}=\sq{\sqrt{A}}$ & a square of side length $a$ and area $a^2=A$;\\
	$\unionsquare{k}$ & the disjoint union of $k$ equal squares of pre-specified area $A$.
\end{tabular}
\end{center}
We refer to Appendix~\ref{sec:interval} for a number of estimates on the eigenvalues $\keig{k}{\interv{a}}{\alpha}$ and $\keig{k}{\rect{a}{A}}{\alpha}$, in particular for small $k$, as well as a description of their asymptotic behaviour in certain parameter ranges.

If $\Omega \subset \R^d$ is any domain, we denote by
\begin{displaymath}
	t\Omega = \{tx: x \in \Omega\}
\end{displaymath}
its homothetic scaling by a factor of $t>0$; then we have the relation $|t\Omega|=t^d|\Omega|$. The eigenvalues of both the Dirichlet and Neumann Laplacians scale well with respect to homothetic scalings of the domain, and this property plays a prominent role in P\'olya's proofs: for any $t>0$ and any $k\geq 1$, we have
\begin{displaymath}
	\keigd{k}{\Omega}=t^{2/d}\keigd{k}{t\Omega} \quad \text{and} \quad \keign{k}{\Omega}=t^{2/d}\keign{k}{t\Omega}.
\end{displaymath}
In the Robin case, we have instead, for any given $\alpha>0$ and $k\geq 1$,
\begin{equation}
\label{eq:homothetic-scaling}
	\keig{k}{\Omega}{\alpha} = t^{2/d}\keig{k}{t\Omega}{\alpha/t},
\end{equation}
cf.~\cite[Section~4.2.1]{bufrke}, also for a discussion of some of the consequences of \eqref{eq:homothetic-scaling}. Before proceeding, for future reference we note how homothetic scalings affect domain minimisation properties.

\begin{lemma}
\label{lem:optimal-scaling}
Fix $A,B>0$. Suppose $\family$ is a family of domains in $\R^2$ such that $|\Omega|=A$ for all $\Omega \in \family$. Consider the family of scaled domains
\begin{displaymath}
	\secfamily := \left\{ B^{-1/2}\Omega:\Omega \in \family \right\}
\end{displaymath}
(so that $|\Omega|=B$ for all $\Omega \in \secfamily$).
\begin{enumerate}
\item If for some $k\geq 1$ and $\alpha_k>0$ there exists $\Omega^A_k \in \family$ such that
\begin{equation}
\label{eq:scaling-opt1}
	\keig{k}{\Omega^A_k}{\alpha_k} = \inf \{ \keig{k}{\Omega}{\alpha_k}: \Omega \in \family\},
\end{equation}
then the scaled domain $\Omega^B_k := B^{-1/2}\Omega^A_k \in \secfamily$ satisfies
\begin{equation}
\label{eq:scaling-opt2}
	\keig{k}{\Omega^B_k}{A^{1/2}B^{-1/2}\alpha_k} = \inf \{ \keig{k}{\Omega}{A^{1/2}B^{-1/2}\alpha_k}: \Omega \in \secfamily\}
\end{equation}
\item If in (1) property \eqref{eq:scaling-opt1} holds for all $\alpha \in (0,\alpha_k]$, then
\begin{displaymath}
	\keig{k}{\Omega^B_k}{\alpha} = \inf \{ \keig{k}{\Omega}{\alpha}: \Omega \in \secfamily\}
\end{displaymath}
for all $\alpha \in (0,A^{1/2}B^{-1/2}\alpha_k]$.
\end{enumerate}
\end{lemma}

\begin{proof}
(1) This follows directly from the scaling relation
\begin{displaymath}
	\keig{k}{t\Omega}{\alpha} = t^{-2}\keig{k}{\Omega}{t\alpha} \geq t^{-2}\keig{k}{\Omega^A_k}{\alpha_k}
\end{displaymath}
for all $\Omega \in \family$, provided $t>0$ and $\alpha>0$ are chosen such that $t\alpha = \alpha_k$. Now choose $t=A^{1/2}B^{-1/2}$ to guarantee the area condition.

(2) follows immediately from (1).
\end{proof}

Likewise, there is still a principle of Wolf--Keller type (cf.~\cite[Section~8]{woke}) which characterises disjoint minimisers among a given family, see \cite[Theorem~2.4]{anfrke}, but in the Robin case things are once again complicated by the scaling relation \eqref{eq:homothetic-scaling}. We will now recall the result from \cite{anfrke} in the form in which we will need it -- actually, while \cite{anfrke} only considered general unions of Lipschitz domains, the result is still true within smaller classes of domains:

\begin{definition}
\label{def:admissible-family}
Let $\family$ be a collection of planar domains and fix $A>0$. We call $\family$ an \emph{admissible family} (for the value $A$, for short simply \emph{admissible}) if every domain in $\family$ has area $A$ and, if $\Omega_1,\ldots,\Omega_n \in \family$ is any finite collection of connected domains in $\family$, then the disjoint union
\begin{displaymath}
	\Omega:= t_1\Omega_1 \cup t_2\Omega_2 \cup \ldots \cup t_n\Omega_n
\end{displaymath}
is in $\family$ whenever the scaling factors $t_1,\ldots,t_n \in [0,1]$ are chosen such that $|\Omega|=A$, i.e., whenever $t_1^2+\ldots+t_n^2=1$.
\end{definition}

Thus an admissible family contains all possible finite disjoint unions of the connected domains in it. For example, the set of all bounded, Lipschitz domains in $\R^2$ of given area $A$ forms an admissible family, as does the set of all finite disjoint unions of rectangles of area $A$, and the set of all finite disjoint unions of disks of area $A$. We will now state a simplified version of \cite[Theorem~2.4]{anfrke} adapted to our needs, noting that the proof of \cite[Theorem~2.4]{anfrke}, ostensibly for all Lipschitz domains, may be repeated verbatim for any admissible family.
\begin{lemma}[Wolf--Keller principle for the Robin problem]
\label{lem:wolf-keller}
Suppose $\family$ is an admissible family of planar domains for some $A>0$ in the sense of Definition~\ref{def:admissible-family} and suppose the disjoint set $\Omega^\ast = \Omega_1 \cup \Omega_2 \in \family$ achieves $\inf \{ \keig{k}{\Omega}{\alpha}: \Omega \in \family \}$ for some $k\geq 2$. Then there exists some $i=1,\ldots,k-1$, as well as scaling factors $t_1$ and $t_2$ with $t_1^2 + t_2^2 = 1$ and numbers $\alpha_1$, $\alpha_2$ such that
\begin{displaymath}
	\Omega_1 = t_1 \Omega_i^\ast, \qquad \Omega_2 = t_2 \Omega_{k-i}^\ast,
\end{displaymath}
with $\Omega_i$ and $\Omega_{k-i}$ realising $\inf \{ \keig{i}{\Omega}{\alpha_1} : t_1^{-1}\Omega \in \family \}$ and $\inf \{ \keig{k-i}{\Omega}{\alpha_2} : t_2^{-1}\Omega \in \family \}$, respectively. Moreover,
\begin{displaymath}
	\keig{k}{\Omega^\ast}{\alpha} = \keig{i}{t_1\Omega_i^\ast}{\alpha} = \keig{k-i}{t_2\Omega_k^\ast}{\alpha}.
\end{displaymath}
\end{lemma}
Obviously, in the above lemma we do not rule out the possibility that $\Omega_1$ and $\Omega_2$ are themselves disconnected, meaning this principle extends inductively to all the connected components of $\Omega^\ast$.

Despite the complicated way in which the Robin problem scales, blowing up a domain via a homothetic scaling always decreases the eigenvalues.

\begin{lemma}
\label{lem:homothetic}
Suppose $\Omega \subset \R^d$ is a bounded, Lipschitz domain, $d\geq 1$, and $\alpha>0$. Then for each $k\geq 1$, the function $t \mapsto \keig{k}{t\Omega}{\alpha}$ is continuous and strictly decreasing in $t \in (0,\infty)$.
\end{lemma}

\begin{proof}
See \cite[Lemma~2.13]{anfrke}.
\end{proof}

One property that the Robin Laplacian does share with its Dirichlet and Neumann counterparts is the fact that on rectangles a complete system of eigenfunctions can be found by separation of variables, as can be shown by the usual means. A particularly important consequence is that the $k^{\rm th}$ Robin eigenvalue of a rectangle is given by a suitable sum of Robin eigenvalues of intervals corresponding to the side lengths of the rectangle. More precisely, in the notation introduced just above, given $A>0$, $\alpha>0$ and $a>0$, for any $k\geq 1$ there exists a pair $(i,j) \in \N \times \N$ such that
\begin{equation}
\label{eq:separation-of-variables}
	\keig{k}{\rect{a}{A}}{\Omega} = \keig{i}{\interv{\sqrt{A}a}}{\alpha} + \keig{j}{\interv{\sqrt{A}/a}}{\alpha};
\end{equation}
moreover, every such pair corresponds to an eigenvalue of $\rect{a}{A}$. Of course, as in the Dirichlet and Neumann cases there is in general no clear relationship between $k$ on the one hand and the pair $(i,j)$ on the other. Partly for this reason, the following definition will be important.

\begin{definition}
\label{def:i-j-mode}
Fix the area $A>0$, the boundary parameter $\alpha>0$ and the side length $a>0$. For any positive integers $i,j$, the eigenvalue of $\rect{a}{A}$ given by
\begin{displaymath}
	\keig{i}{\interv{\sqrt{A}a}}{\alpha} + \keig{j}{\interv{\sqrt{A}/a}}{\alpha}.
\end{displaymath}
will be denoted by $\eigmode{i}{j}{\rect{a}{A}}{\alpha}$ and called \emph{the eigenvalue (of $\rect{a}{A}$) associated with the $(i,j)$ mode}.
\end{definition}

\begin{remark}
\label{rem:i-j-mode}
By standard Sturm--Liouville theory, the eigenvalue $\keig{k}{\interv{a}}{\alpha}$ is always simple, and its eigenfunction has exactly $k-1$ zeros in the interior of
$\interv{a}$, that is, it has $k$ nodal domains. Thus there is always exactly one eigenfunction (up to scalar multiples) associated with the eigenmode $\eigmode{i}{j}{\rect{a}{A}}{\alpha}$ (even if the corresponding eigenvalue itself has higher multiplicity), and the $i\times j$ nodal domains of the eigenfunction are rectangles arranged in a grid pattern, just as in the Dirichlet and Neumann cases, with $\eigmode{i}{j}{\rect{a}{A}}{\alpha}$ being the first eigenvalue of the nodal domain with Dirichlet conditions on the edges interior to $\Omega$, and the Robin condition on those edges it has in common with $\partial\Omega$ (``exterior edges''). This means that not all nodal domains will be isometric copies of each other: the area of a nodal domain is a strictly decreasing function of its number of exterior edges; furthermore, any two nodal domains with the same number of exterior edges must be isometric to each other. This follows from the fact that having Robin boundary conditions on a side lowers the eigenvalue, together with the monotonicity with respect to homothetic scalings.
\end{remark}

We finish this section by noting the following continuity result for the eigenvalues with respect to edge lengths.

\begin{lemma}
\label{lem:length-continuity}
Fix $A>0$ and $\alpha>0$. Then for each $k\geq 1$, the maps
\begin{displaymath}
	a \mapsto \keig{k}{\interv{a}}{\alpha},\qquad a \mapsto \keig{k}{\rect{a}{A}}{\alpha}
\end{displaymath}
are continuous in $a\geq 1$, the former even being analytic.
\end{lemma}

\begin{proof}
The continuity of $a \mapsto \keig{k}{\interv{a}}{\alpha}$ is just the continuity of the mapping $t \mapsto t\Omega$ in the special case $\Omega = \interv{}$ (Lemma~\ref{lem:homothetic}); the analyticity follows from the fact that the eigenvalues are given as solutions of transcendental equations in $\tan$ (or $\cot$) which are analytic functions of their parameters, and each eigenvalue is simple. For the continuity of $a \mapsto \keig{k}{\rect{a}{A}}{\alpha}$, use the continuity of $a \mapsto \keig{k}{\interv{a}}{\alpha}$ together with the representation \eqref{eq:separation-of-variables}, noting that for any fixed $k$, the set of values of $a$ for which the relationship $k \sim (i,j)$ changes obviously consists of isolated points, and each eigenvalue is continuous across each isolated point.
\end{proof}

\section{Existence of minimising rectangles and unions of rectangles}
\label{sec:rectangles}

Let us start by giving a basic result stating that the problems we are considering are well posed: for any fixed eigenvalue and boundary parameter, there is a rectangle minimising that eigenvalue among all rectangles of given area; the same is true if we replace ``rectangle'' by ``union of rectangles''.

\begin{theorem}
\label{thm:existence}
Fix $k\geq 1$, $\alpha >0$ and $A>0$. Then there exists a rectangle $\mathcal{R}^\ast = \mathcal{R}^\ast (k,A,\alpha)$ of area $A$ such that
\begin{displaymath}
	\keig{k}{\mathcal{R}^\ast}{\alpha} = \keigoptrect{k}{A}{\alpha} = \inf \{ \keig{k}{\mathcal{R}}{\alpha}: \mathcal{R} 
	\textrm{ is a rectangle of area }A \}.
\end{displaymath}
Moreover, there exists a disjoint union of rectangles $\Omega^\ast = \Omega^\ast (k,A,\alpha)$ of total area $A$ such that
\begin{displaymath}
	\keig{k}{\Omega^\ast}{\alpha} = \keigopt{k}{A}{\alpha} = \inf\{ \keig{k}{\Omega}{\alpha}: \Omega
	\textrm{ disjoint union of rectangles, } |\Omega|=A\}.
\end{displaymath}
\end{theorem}

\begin{remark}
\label{rem:dirichlet-existence}
If we consider instead the Dirichlet Laplacian on rectangles and unions of rectangles, i.e., if we consider $\keigoptrect{k}{A}{\infty}$ and $\keigopt{k}{A}{\infty}$, then for each given $k\geq 1$ and $A>0$ we can also obtain a minimising domain in each case. We omit the proof, which is an especially easy simplified version of the proof of Theorem~\ref{thm:existence}.
\end{remark}

\begin{proof}[Proof of Theorem~\ref{thm:existence}]
1. First we consider the case of rectangles. Denote by $\rect{a_n}{A}$ with $a_n \geq 1$ a minimising sequence for $\keigoptrect{k}{A}{\alpha}$. Since
\begin{displaymath}
	\keig{k}{\rect{a}{A}}{\alpha} \geq \keig{1}{\rect{a}{A}}{\alpha} = \eigmode{1}{1}{\rect{a}{A}}{\alpha}
	\geq \keig{1}{\interv{A^{1/2}/a}}{\alpha} \to \infty
\end{displaymath}
as $a \to \infty$ (cf.~Proposition~\ref{prop:firsteiginterv}), there exists some $\tilde a\geq 1$ such that $a_n \leq \tilde a$ for all $n\in \N$. Thus there exists an
$a^\ast \geq 1$ such that $a_n \to a^\ast$ up to a subsequence. The corresponding rectangle $\rect{a^\ast}{A}$ has area $A$ and, by Lemma~\ref{lem:length-continuity},
\begin{displaymath}
	\keig{k}{\rect{a_n}{A}}{\alpha} \to \keig{k}{\rect{a^\ast}{A}}{\alpha}.
\end{displaymath}

2. Now we consider unions of rectangles. Suppose $\Omega_n$ is a minimising sequence for $\keigopt{k}{A}{\alpha}$. Since the eigenvalues are monotonic with respect to homothetic scalings of the domain (see Lemma~\ref{lem:homothetic}), we may assume without loss of generality that each domain $\Omega_n$ has no more than $k$ connected components, and that each connected component $U$ of $\Omega_n$ is ``needed'' (in the sense that $\keig{1}{U}{\alpha} \leq \keig{k}{\Omega_n}{\alpha}$ and $\keig{k}{\Omega_n \setminus U}{\alpha} > \keig{k}{\Omega_n}{\alpha}$ for each $U$).

By the pigeonhole principle, there exists some $\ell \leq k$ and subsequence, whose members we shall still denote by $\Omega_n$, such that each $\Omega_n$ has exactly $\ell$ connected components $U_n^1,\ldots,U_n^\ell$, and such that for each $i=1,\ldots,\ell$ there exists a fixed $j=j(i)$ such that $\keig{j}{U_n^i}{\alpha} \leq \keig{k}{\Omega_n}{\alpha} < \keig{j+1}{U_n^i}{\alpha}$ (in words, each component $U_n^i$ always ``contributes'' the same number $j$ of eigenvalues to the first $k$ of $\Omega_n$, independently of $n$). Applying the argument for rectangles, part 1, to each connected component, we obtain a limit domain $\Omega^\ast$ of area $A$ such that
\begin{displaymath}
	\keig{k}{\Omega_n}{\alpha} \to \keig{k}{\Omega^\ast}{\alpha}
\end{displaymath}
as $n \to \infty$.
\end{proof}

\section{Isoperimetric inequalities for the low eigenvalues}
\label{sec:isoperimetric}

We will next prove that the square minimises the first eigenvalue among all rectangles (or indeed their unions) of given area; this may be considered an inequality of ``isoperimetric'' type, since the square has the least perimeter among all such domains.

\begin{theorem}
\label{thm:lambda1}
Let $A>0$ be given. Then for any $\alpha>0$ and any finite union of disjoint rectangles $\Omega \subset \R^2$ of total area $A$, we have
\begin{displaymath}
	\keig{1}{\Omega}{\alpha} \geq \keig{1}{\sq{\sqrt{A}}}{\alpha},
\end{displaymath}
with equality if and only if $\Omega$ is itself a square of side length $\sqrt{A}$.
\end{theorem}

\begin{proof}
Since $\alpha > 0$ is arbitrary, we may assume without loss of generality that $A=1$ (cf.~Lemma~\ref{lem:optimal-scaling}).

1. We start by proving the statement for rectangles. So fix $a \geq 1$ and consider $\rect{a}{1}$. By separation of variables, cf.~\eqref{eq:separation-of-variables},
\begin{displaymath}
	\keig{1}{\rect{a}{1}}{\alpha} = \keig{1}{\interv{a}}{\alpha} + \keig{1}{\interv{a^{-1}}}{\alpha},
\end{displaymath}
where we recall $\keig{1}{\interv{b}}{\alpha}$ is the smallest positive solution $\lambda$ of the equation~\eqref{eq:1eig-interval}.
To prove that $\keig{1}{\rect{a}{1}}{\alpha}$ achieves a unique global minimum at $a=1$, it thus suffices to show that for any $a > 1$,
\begin{displaymath}
	\frac{\partial}{\partial a}\keig{1}{\interv{a}}{\alpha} \leq 0 \leq \frac{\partial}{\partial a}\keig{1}{\interv{a^{-1}}}{\alpha} \ \text{and} \
	\left|\frac{\partial}{\partial a}\keig{1}{\interv{a}}{\alpha}\right| < \left|\frac{\partial}{\partial a}\keig{1}{\interv{a^{-1}}}{\alpha}\right|.
\end{displaymath}
Differentiating $\lambda$ implicitly with respect to $a$ in \eqref{eq:1eig-interval}, a slightly tedious but elementary calculation leads us to
\begin{displaymath}
	\frac{\partial}{\partial a}\keig{1}{\interv{a}}{\alpha} = -\frac{2\keig{1}{\interv{a}}{\alpha}\alpha}
	{2\sin^2\left(\frac{a}{2}{\sqrt{\keig{1}{\interv{a}}{\alpha}}}\right)+a\alpha} < 0.
\end{displaymath}
A similar calculation yields
\begin{displaymath}
	\frac{\partial}{\partial a}\keig{1}{\interv{a^{-1}}}{\alpha} = \frac{2\keig{1}{\interv{a^{-1}}}{\alpha}\alpha}
	{2a^2\sin^2\left(\frac{1}{2a}{\sqrt{\keig{1}{\interv{a^{-1}}}{\alpha}}}\right)+a\alpha} > 0.
\end{displaymath}
Now the scaling relations \eqref{eq:homothetic-scaling} and the inequality $\keig{1}{\interv{1}}{a\alpha} \leq a\keig{1}{\interv{1}}{\alpha}$ for $a\geq 1$
(which follows from the last assertion in Lemma~\ref{lem:parameter-continuity}) give
\begin{equation}
\label{eq:1eig-interval-a-scaling-est}
	\keig{1}{\interv{a}}{\alpha} = a^{-2}\keig{1}{\interv{1}}{a\alpha} \leq a^{-1}\keig{1}{\interv{1}}{\alpha},
\end{equation}
while the reverse inequality for $a^{-1}<1$, that is, $\keig{1}{\interv{1}}{a^{-1}\alpha} \geq a^{-1}\keig{1}{\interv{1}}{\alpha}$, implies
\begin{equation}
\label{eq:1eig-interval-a-1-scaling-est}
	\keig{1}{\interv{a^{-1}}}{\alpha} = a^2\keig{1}{\interv{1}}{a^{-1}\alpha} \geq a\keig{1}{\interv{1}}{\alpha}.
\end{equation}
Applying \eqref{eq:1eig-interval-a-scaling-est} and \eqref{eq:1eig-interval-a-1-scaling-est} to the expressions for the derivatives found above, we have
\begin{displaymath}
\begin{aligned}
	\left|\frac{\frac{\partial}{\partial a}\keig{1}{\interv{a}}{\alpha}}{\frac{\partial}{\partial a}\keig{1}{\interv{a^{-1}}}{\alpha}}\right|
	&=\frac{\keig{1}{\interv{a}}{\alpha}}{\keig{1}{\interv{a^{-1}}}{\alpha}}
	\frac{2a^2\sin^2\left(\frac{1}{2a}{\sqrt{\keig{1}{\interv{a^{-1}}}{\alpha}}}\right)+a\alpha}
	{2\sin^2\left(\frac{a}{2}{\sqrt{\keig{1}{\interv{a}}{\alpha}}}\right)+a\alpha}\\
	&\leq \frac{2a\sin^2\left(\frac{1}{2a}{\sqrt{\keig{1}{\interv{a^{-1}}}{\alpha}}\alpha}\right)+\alpha}
	{2a\sin^2\left(\frac{a}{2}{\sqrt{\keig{1}{\interv{a}}{\alpha}}}\right)+a^2\alpha}.
\end{aligned}
\end{displaymath}
To complete the proof for rectangles it suffices to show that this expression is smaller than $1$ whenever $a>1$; in this case, it is in turn sufficient to show that
\begin{equation}
\label{eq:sine-comparison}
	\sin^2\left(\frac{1}{2a}{\sqrt{\keig{1}{\interv{a^{-1}}}{\alpha}}}\right) \leq 
	\sin^2\left(\frac{a}{2}{\sqrt{\keig{1}{\interv{a}}{\alpha}}}\right).
\end{equation}
But since $\keig{1}{\interv{b}}{\alpha}$ is always smaller than the corresponding Dirichlet eigenvalue $\pi^2/b^2$ for any $b>0$, the arguments of the sines in \eqref{eq:sine-comparison} are always less than $\pi/2$. In particular, since in this range $x\mapsto \sin^2(x)$ is monotonically increasing in $x$, to establish \eqref{eq:sine-comparison} it is sufficient to show that
\begin{displaymath}
	\frac{1}{2a}{\sqrt{\keig{1}{\interv{a^{-1}}}{\alpha}}} \leq \frac{a}{2}{\sqrt{\keig{1}{\interv{a}}{\alpha}}},
\end{displaymath}
which, upon rearrangement, is equivalent to
\begin{displaymath}
	\keig{1}{\interv{a^{-1}}}{\alpha} \leq a^4 \keig{1}{\interv{a}}{\alpha}.
\end{displaymath}
But this now follows from the scaling relations $\keig{1}{\interv{a^{-1}}}{\alpha} \leq a^2 \keig{1}{\interv{1}}{\alpha}$ and $\keig{1}{\interv{a}}{\alpha} \geq a^{-2} \keig{1}{\interv{1}}{\alpha}$ (cf.~\eqref{eq:homothetic-scaling}). This establishes \eqref{eq:sine-comparison} and hence the statement of the theorem for rectangles.

2. Now suppose that $\Omega$ is a union of two or more rectangles. Then there exists some rectangle $\rect{a_1}{A_1}$ with $A_1<1=A$, such that
\begin{displaymath}
	\keig{1}{\Omega}{\alpha} = \keig{1}{\rect{a_1}{A_1}}{\alpha}.
\end{displaymath}
Using what we have just shown for rectangles, the inequality $A_1<1$, and the fact that $\keig{1}{t\Omega}{\alpha}$ is a strictly monotonically decreasing function of $t>0$ for any bounded, Lipschitz domain $\Omega$ (see Lemma~\ref{lem:homothetic}),
\begin{displaymath}
	\keig{1}{\rect{a_1}{A_1}}{\alpha} \geq \keig{1}{\sq{\sqrt{A_1}}}{\alpha} > \keig{1}{\sq{1}}{\alpha},
\end{displaymath}
which proves the theorem for $\Omega$.
\end{proof}

Theorem~\ref{thm:lambda1} yields as a corollary a corresponding statement concerning the second eigenvalue: that it is always minimised by the union $\unionsquare{2}$ of two equal squares. This statement of Hong--Krahn--Szego type can be proved by the usual means.

\begin{corollary}
\label{cor:lambda2}
Fix $A>0$. Then for any $\alpha>0$ and any finite union of disjoint rectangles $\Omega \subset \R^2$ of total area $A$,
\begin{displaymath}
	\keig{2}{\Omega}{\alpha} \geq \keig{2}{\unionsquare{2}}{\alpha},
\end{displaymath}
where $\unionsquare{2}$ is the disjoint union of two equal squares, each of side length $\sqrt{A/2}$. Equality holds if and only if $\Omega = \unionsquare{2}$ up to rigid transformations.
\end{corollary}

\begin{proof}
Again, it suffices to prove the statement for $A=1$.

1. Suppose first $\Omega$ is a rectangle, say $\rect{a}{1}$ for some $a \geq 1$, which we assume to be centred at the origin. Then the zero (nodal) set of an eigenfunction corresponding to $\keig{2}{\rect{a}{1}}{\alpha}$ is given by the set $\{(0,y): y \in (-1/(2a), 1/(2a)) \}$, and $\keig{2}{\rect{a}{1}}{\alpha}$ is equal to the first eigenvalue of the Laplacian on the rectangle $(0,a/2) \times (-1/(2a), 1/(2a)) \simeq \rect{a/\sqrt{2}}{1/2}$ with Dirichlet conditions on the side $\{(0,y): y \in (-1/(2a), 1/(2a)) \}$ and Robin conditions with boundary coefficient $\alpha$ on the other three sides. Using the restriction of the eigenfunction for $\keig{2}{\rect{a}{1}}{\alpha}$ to either of its nodal domains $\rect{a/\sqrt{2}}{1/2}$ as a test function for the problem on $\rect{a/\sqrt{2}}{1/2}$ with Robin boundary conditions on all four sides yields
\begin{displaymath}
	\keig{2}{\rect{a}{1}}{\alpha} > \keig{1}{\rect{a/\sqrt{2}}{1/2}}{\alpha}.
\end{displaymath}
By our isoperimetric inequality, Theorem~\ref{thm:lambda1},
\begin{displaymath}
	\keig{1}{\rect{a/\sqrt{2}}{1/2}}{\alpha} \geq \keig{1}{\sq{\sqrt{1/2}}}{\alpha} = \keig{2}{\unionsquare{2}}{\alpha},
\end{displaymath}
und thus the corollary is true if $\Omega$ is a rectangle.

2. Suppose now that $\Omega$ is a union of at least two rectangles. There are two cases consider: (1) there exists a rectangle $\rect{a_1}{A_1}$ with $A_1<1$ such that $\keig{2}{\Omega}{\alpha} = \keig{2}{\rect{a_1}{A_1}}{\alpha}$; or (2) there exist two rectangles $\rect{a_2}{A_2}$ and $\rect{a_3}{A_3}$ belonging to $\Omega$, such that
\begin{displaymath}
	\keig{2}{\Omega}{\alpha} = \max \{ \keig{1}{\rect{a_2}{A_2}}{\alpha}, \keig{1}{\rect{a_3}{A_3}}{\alpha} \}.
\end{displaymath}

For case (1), apply our result for rectangles proved just above directly to $\rect{a_1}{A_1}$ and use that the eigenvalues are (strictly) monotonically decreasing with respect to homothetic scalings.

For case (2), applying Theorem~\ref{thm:lambda1} to each of $\rect{a_2}{A_2}$ and $\rect{a_3}{A_3}$ separately,
\begin{displaymath}
	\keig{2}{\Omega}{\alpha} \geq \max \{ \keig{1}{\sq{\sqrt{A_2}}}{\alpha}, \keig{1}{\sq{\sqrt{A_3}}}{\alpha} \}.
\end{displaymath}
This maximum is at least as large as $\keig{1}{\sq{\sqrt{1/2}}}{\alpha} = \keig{2}{\unionsquare{2}}{\alpha}$ since at least one of $A_2$ and $A_3$ is no larger than $1/2$. 
For strictness of the inequality in this case, assuming $\Omega$ not to be equal to $\unionsquare{2}$, if it has at least three connected components we may discard the superfluous one(s) and inflate $\rect{a_2}{A_2}$ and $\rect{a_3}{A_3}$ to decrease $\keig{1}{\Omega}{\alpha}$ strictly. So assume $\Omega = \rect{a_2}{A_2} \cup \rect{a_3}{A_3}$. Then either one of these rectangles is not a square, in which case Theorem~\ref{thm:lambda1} yields strict inequality, or one of them has area strictly less than $1/2$, in which case it follows from the assertion on strictness in Lemma~\ref{lem:homothetic}.
\end{proof}

Note that for the above argument it was important that the nodal domains associated with the second eigenfunction of a rectangle are themselves other rectangles, so that Theorem~\ref{thm:lambda1} is applicable.

\section{Convergence to the Dirichlet minimisers as $\alpha \to \infty$}
\label{sec:convergence-to-dirichlet}

It is well known and easy to show that if $\Omega \subset \R^2$ is any fixed domain, then, for any fixed $k\geq 1$, we have $\keig{k}{\Omega}{\alpha} \to \keig{k}{\Omega}{\infty}$ as $\alpha\to\infty$, where we recall that $\keig{k}{\Omega}{\infty}$ is the $k^{\rm th}$ Dirichlet eigenvalue (cf., e.g., \cite[Proposition~4.5]{bufrke}). Before we turn to the behaviour of the optimal values for small $\alpha>0$, or fixed $\alpha>0$ and large $k$, we will show that for any fixed $k\geq 1$ and $A>0$, we also have convergence of the optimal values to their Dirichlet counterparts as $\alpha \to \infty$. We note that this does not follow immediately from the convergence for each \emph{fixed} domain since, in general, the optimisers depend on $\alpha>0$.

\begin{theorem}
\label{thm:robin-dirichlet-convergence}
Fix $k\geq 1$ and $A>0$. Then, as $\alpha \to \infty$,
\begin{equation}
\label{eq:robin-dirichlet-convergence}
	\keigoptrect{k}{A}{\alpha} \to \keigoptrect{k}{A}{\infty} \qquad \text{and} \qquad \keigopt{k}{A}{\alpha} \to \keigopt{k}{A}{\infty}.
\end{equation}
Moreover, if $\alpha_n \to \infty$, then
\begin{enumerate}
\item if $\rect{a_n^\ast}{A}$ is any sequence of rectangles realising $\keigoptrect{k}{A}{\alpha}$, then up to a subsequence $a_n^\ast \to a^\ast$, where $\rect{a^\ast}{A}$ realises $\keigoptrect{k}{A}{\infty}$;
\item if $\alpha_n \to \infty$ and $\Omega_n^\ast$ realises $\keigopt{k}{A}{\alpha}$, then there exists some $\Omega^\ast$ realising $\keigopt{k}{A}{\alpha}$ such that, up to a subsequence, $\Omega_n^\ast$ and $\Omega^\ast$ all have the same number of connected components, and, if these are numbered appropriately, then statement (1) holds for each of them separately.
\end{enumerate}
\end{theorem}

The necessity for considering subsequences in the above theorem comes from the fact that the minimisers may not be unique for each fixed $k\geq 1$ and $\alpha \in (0,\infty]$ (indeed, in general this seems to be unknown).

One of the key tools in the proof is the following lemma, which will also play an important role in subsequent sections. It gives us control over long, thin rectangles by giving us an explicit estimate on the long side length of the rectangle necessary to ensure that the $k^{\rm th}$ eigenvalue corresponds to the $(k,1)$ mode (see Definition~\ref{def:i-j-mode}); this, in turn, can be estimated fairly explicitly using the bounds in Appendix~\ref{sec:interval}.

\begin{lemma}
\label{lem:k-1-mode-est}
Let $k \geq 1$, $A>0$, $\alpha>0$ and $a\geq 1$. Then $\keig{k}{\rect{A}{a}}{\alpha} = \eigmode{k}{1}{\rect{A}{a}}{\alpha}$ whenever
\begin{equation}
\label{eq:a-est}
	a \geq k^{1/2}.
\end{equation}
\end{lemma}

\begin{proof}
It suffices to show that \eqref{eq:a-est} implies
\begin{displaymath}
	\eigmode{k}{1}{\rect{A}{a}}{\alpha} \leq \eigmode{1}{2}{\rect{A}{a}}{\alpha}.
\end{displaymath}
Now by definition of the $(1,2)$ mode,
\begin{displaymath}
	\eigmode{1}{2}{\rect{A}{a}}{\alpha} = \keig{1}{\interv{A^{1/2}a}}{\alpha} + \keig{2}{\interv{A^{1/2}a^{-1}}}{\alpha} 
	\geq \keig{2}{\interv{A^{1/2}a^{-1}}}{\alpha},
\end{displaymath}
while
\begin{displaymath}
	\eigmode{k}{1}{\rect{A}{a}}{\alpha} = \keig{k}{\interv{A^{1/2}a}}{\alpha} + \keig{1}{\interv{A^{1/2}a^{-1}}}{\alpha} 
	\leq \keig{k}{\interv{A^{1/2}a}}{\infty} + \keig{1}{\interv{A^{1/2}a^{-1}}}{\alpha}.
\end{displaymath}
We now invoke the estimate on the Fundamental Gap
\begin{displaymath}
	\keig{2}{\interv{D}}{\alpha} - \keig{1}{\interv{D}}{\alpha} \geq \frac{\pi^2}{D^2}
\end{displaymath}
valid for the Robin Laplacian in one dimension (or more generally on any domain whose first Robin eigenfunction is log-concave, see \cite[Theorem~2.1]{anclha}), with the choice $D=A^{1/2}a^{-1}$. We thus have
\begin{displaymath}
	\eigmode{k}{1}{\rect{A}{a}}{\alpha} - \eigmode{1}{2}{\rect{A}{a}}{\alpha} 
	\leq \keig{k}{\interv{A^{1/2}a}}{\infty} - \frac{\pi^2 a^2}{A}
	= \frac{\pi^2 k^2}{A a^2} - \frac{\pi^2 a^2}{A}.
\end{displaymath}
This is non-positive as long as $a\geq k^{1/2}$.
\end{proof}

\begin{proof}[Proof of Theorem~\ref{thm:robin-dirichlet-convergence}]
Fix $k\geq 1$ and an arbitrary sequence $\alpha_n \to \infty$; obviously, it suffices to prove the theorem for this sequence.

1. We start with rectangles. We first claim the existence of an $\hat a\geq 1$ such that for all $n \geq 1$ the optimal rectangle $\rect{a_n^\ast}{A}$ satisfies $a_n^\ast \leq \hat a$, i.e., the sequence of optimal rectangles is uniformly bounded in $n$. In fact, by Lemma~\ref{lem:k-1-mode-est} and the lower bound in \eqref{eq:eig1ineq1}, if $a \geq k^{1/2}$, then we have
\begin{displaymath}
\begin{split}
	\keig{k}{\rect{a_n^\ast}{A}}{\alpha} = \eigmode{k}{1}{\rect{a_n^\ast}{A}}{\alpha} &\geq \keig{1}{\interv{A^{1/2}a^{-1}}}{\alpha}\\
	&\geq \frac{2\pi^2\alpha}{A^{1/2}a^{-1}(\pi^2 + 2A^{1/2}\alpha a^{-1})} \to \infty
\end{split}
\end{displaymath}
uniformly in $\alpha \geq \alpha_1 > 0$ as $a \to \infty$. This proves the claim.

2. Since $(a_n^\ast)_{n\geq 1}$ is bounded, up to a subsequence we have $a_n^\ast \to \tilde{a}$ for some $\tilde{a} \in [1,\hat a]$. We claim that the corresponding rectangle $\rect{\tilde{a}}{A}$ realises $\keigoptrect{k}{A}{\infty}$. Now by the pigeonhole principle, up to another sequence, for each $\ell = 1,\ldots,k$ there exist $1\leq i_\ell,j_\ell\leq \ell$ such that
\begin{displaymath}
	\keig{\ell}{\rect{a_n^\ast}{A}}{\alpha} = \eigmode{i_\ell}{j_\ell}{\rect{a_n^\ast}{A}}{\alpha} = \keig{i_\ell}{\interv{A^{1/2}a_n^\ast}}{\alpha}
	+ \keig{j_\ell}{\interv{A^{1/2}(a_n^\ast)^{-1}}}{\alpha}
\end{displaymath}
for all $n$.

3. We claim that for any $k\geq 1$ and any sequences of numbers $a_n\geq 1$ such that $a_n \to a$ and $\alpha_n \to \infty$, we have
\begin{equation}
\label{eq:1-d-diag-conv}
	\keig{k}{\interv{a_n}}{\alpha_n} \to \keig{k}{\interv{a}}{\infty}.
\end{equation}
To prove \eqref{eq:1-d-diag-conv}, we rescale, cf.~\eqref{eq:homothetic-scaling}: writing $\interv{a} = \frac{a}{a_n}\interv{a_n}$,
\begin{displaymath}
	\keig{k}{\interv{a_n}}{\alpha_n} = \left(\frac{a}{a_n}\right)^2 \keig{k}{\interv{a}}{\alpha_n a_n/a}.
\end{displaymath}
Since $a/a_n \to 1$ and $\alpha_n a_n/a \to \infty$ as $n\to \infty$, it follows from Lemma~\ref{lem:length-continuity} that
\begin{displaymath}
	\keig{k}{\interv{a}}{\alpha_n a_n/a} \to \keig{k}{\interv{a}}{\infty},
\end{displaymath}
which proves the claim.

4. Combining Steps 3 and 4, we conclude that
\begin{displaymath}
	\keig{\ell}{\rect{a_n^\ast}{A}}{\alpha} = \eigmode{i_\ell}{j_\ell}{\rect{a_n^\ast}{A}}{\alpha}\to 
	\eigmode{i_\ell}{j_\ell}{\rect{\tilde{a}}{A}}{\infty}
\end{displaymath}
for all $\ell=1,\ldots,k$. y induction on $\ell$, we also obtain $\eigmode{i_\ell}{j_\ell}{\rect{\tilde{a}}{A}}{\infty} = \keig{\ell}{\rect{\tilde{a}}{A}}{\infty}$ and in particular $\keig{k}{\rect{a_n^\ast}{A}}{\alpha} \to \keig{k}{\rect{\tilde{a}}{A}}{\infty} $. Since
\begin{displaymath}
	\keig{k}{\rect{a_n^\ast}{A}}{\alpha} = \keigoptrect{k}{A}{\alpha} \leq \keigoptrect{k}{A}{\infty} \leq \keig{k}{\rect{\tilde{a}}{A}}{\infty},
\end{displaymath}
the first inequality following since the same is true of any fixed domain, we thus have $\keigoptrect{k}{A}{\infty} = \keig{k}{\rect{\tilde{a}}{A}}{\infty}$, and $\rect{\tilde{a}}{A}$ is a minimiser. Moreover, since we have shown that every sequence $\alpha_n \to \infty$ has a subsequence for which $\keigoptrect{k}{A}{\alpha_n} \to \keigoptrect{k}{A}{\infty}$ for this subsequence, the hair-splitting lemma implies the convergence of the whole sequence.

5. Finally, we deal with unions of rectangles. We assume that $\Omega_n^\ast$, not necessarily connected, realises $\keigopt{k}{A}{\alpha_n}$. Up to a subsequence each $\Omega_n^\ast$ has some fixed number $m\geq 1$ of connected components (i.e., rectangles) $U_{1,n},\ldots, U_{m,n}$ and (up to a further subsequence and a possible relabelling of the $U_{j,n}$) there exist numbers $i_1,\ldots,i_m$ such that $i_1+\ldots +i_m = k$ and
\begin{displaymath}
	\keig{k}{\Omega_n^\ast}{\alpha_n} = \keig{i_1}{U_{1,n}}{\alpha_n}=\ldots =\keig{i_m}{U_{m,n}}{\alpha_n}
\end{displaymath}
for all $n$ (see Lemma~\ref{lem:wolf-keller}). Now the argument of Steps 2 and 3, applied to each of the connected components, implies the existence of a domain
\begin{displaymath}
	\Omega^\ast = U_1 \cup \ldots \cup U_m
\end{displaymath}
for rectangles $U_1,\ldots,U_m$, such that $|\Omega^\ast|=A$ and, up to a further subsequence,
\begin{displaymath}
	\keig{\ell}{\Omega_n^\ast}{\alpha_n} \to \keig{\ell}{\Omega^\ast}{\infty}
\end{displaymath}
as $n\to \infty$ for each $\ell=1,\ldots,k$, since for each connected component we can find a further subsequence for which this is true for that connected component. In particular,
\begin{displaymath}
	\keig{k}{A}{\infty} \geq \keigopt{k}{A}{\alpha_n} = \keig{k}{\Omega_n^\ast}{\alpha_n} \to \keig{k}{\Omega^\ast}{\infty}
	\geq \keig{k}{A}{\infty},
\end{displaymath}
implying the optimality of $\Omega^\ast$. Moreover, the same argument as before using the hair-splitting lemma implies $\keigopt{k}{A}{\alpha_n} \to \keig{k}{A}{\infty}$ for the whole sequence.
\end{proof}

\section{Optimal rectangles: Proof of Theorem~\ref{thm:rectangles}}
\label{sec:opt-rect}

In this section we prove that $\keigoptrect{k}{A}{\alpha}$ grows like $k^{2/3}$ for fixed $\alpha$, at the same time giving asymptotically sharp two-sided bounds. The argument consists of two parts: firstly, we obtain the desired estimate for $(k,1)$ modes (cf.~Definition~\ref{def:i-j-mode}); then we show that for $k$ large enough (depending on $\alpha$ and $A$) the $k^{\rm th}$ eigenvalue of the optimal rectangle is given by its $(k,1)$ mode. For this, we will need to introduce and give a rough estimate on the eigenvalue counting function of the Robin Laplacian on a fixed domain. We present each part in a separate subsection.

\subsection{Two-sided bounds on the $(k,1)$ mode}

We start with the $(k,1)$ mode. Note that the bounds in the following estimate correspond exactly to those in Theorem~\ref{thm:rectangles}, although we have them for a different range of $A$, $\alpha$, $k$. Notationally, we set
\begin{equation}
\label{eq:k-1-mode-opt}
	\eigmodeoptrect{k}{1}{A}{\alpha} := \inf \left\{ \eigmode{k}{1}{\rect{a}{A}}{\alpha} : a\geq 1\right\}
\end{equation}
to be the smallest value attainable by a $(k,1)$ mode.

\begin{lemma}
\label{lem:k-1-mode-bounds}
For any $A>0$, $\alpha>0$ and $k\geq 2$, we have the bounds
\begin{equation}
\label{eq:k-1-mode-bounds}
	\fr{3\pi^2\alpha^{2/3}}{\left(\pi^2+2A^{1/2}\alpha\right)^{2/3}A^{2/3}}(k-2)^{2/3} \leq \eigmodeoptrect{k}{1}{A}{\alpha} 
	\leq 3\left(\fr{\pi\alpha}{A}\right)^{2/3} k^{2/3},
\end{equation}
the upper bound holding provided $\alpha \leq \pi^2 A^{-1/2}k^2$. Moreover, the infimum in \eqref{eq:k-1-mode-opt} is attained by a rectangle whose side length is proportional to $k^{2/3}$ for large $k$ and fixed $A,\alpha>0$. Finally,
\begin{equation}
\label{eq:k-1-opt-asymptotics}
	\lim_{k\to\infty} \frac{\eigmodeoptrect{k}{1}{A}{\alpha}}{k^{2/3}} = 3\left(\fr{\pi\alpha}{A}\right)^{2/3}.
\end{equation}
\end{lemma}

\begin{proof}
By definition of the $(k,1)$ mode, we have
\begin{equation}
\label{eq:k-1-breakdown}
\begin{split}
	\eigmode{k}{1}{\rect{a}{A}}{\alpha} &= \keig{k}{\interv{A^{1/2}a}}{\alpha} + \keig{1}{\interv{A^{1/2}a^{-1}}}{\alpha}\\
	&= \keig{1}{\interv{A^{1/2}\tilde a}}{\infty} + \keig{1}{\interv{A^{1/2}a^{-1}}}{\alpha}
\end{split}
\end{equation}
for some $\tilde a \in \left[\frac{a}{k},\frac{a}{k-2}\right]$, where the second equality comes about from restricting to any one of the identical $k-2$ nodal domains of the corresponding eigenfunction which do not touch the shorter sides of the rectangle (cf.~Remark~\ref{rem:i-j-mode}). For the upper bound, we use the monotonicity of the Dirichlet eigenvalue with respect to shrinking the interval and the bound \eqref{eq:main-interval-bound}, applied to $\keig{1}{\interv{A^{1/2}\tilde a}}{\infty}$ and $\keig{1}{\interv{A^{1/2}a^{-1}}}{\alpha}$, respectively, to obtain
\begin{displaymath}
	\eigmode{k}{1}{\rect{a}{A}}{\alpha} \leq \frac{\pi^2 k^2}{Aa^2} + \frac{2\alpha a}{A^{1/2}}.
\end{displaymath}
We now make the \emph{Ansatz} $a=c_1 k^{2/3}$ (for some $c_1>0$ which may \emph{a priori} depend on $k$, i.e., formally, we take $c_1:=ak^{-2/3}$); then, switching to considering the infimum over all $a\geq 1$,
\begin{displaymath}
	\eigmodeoptrect{k}{1}{A}{\alpha} \leq \inf_{c_1\geq k^{-2/3}} \pi^2k^{2/3}\left[ \frac{1}{Ac_1^2}+\frac{2\alpha}{\pi^2 A^{1/2}}c_1\right].
\end{displaymath}
The infimum over $c_1>0$ is obtained independently of $k\geq 1$ at
\begin{equation}
\label{eq:c1}
	c_1=\pi^{2/3}A^{-1/6}\alpha^{-1/3},
\end{equation}
resulting in a right-hand side of value
\begin{displaymath}
	3\pi^{2/3}A^{-2/3}\alpha^{2/3}k^{2/3};
\end{displaymath}
this is valid provided this minimum occurs when $a=c_1 k^{2/3} \geq 1$, that is,
\begin{displaymath}
	\pi^{2/3}A^{-1/6}\alpha^{-1/3}k^{2/3} \geq 1,
\end{displaymath}
which after simplification reduces to $\alpha \leq \pi A^{-1/2}k^2$. Observe also that the choice of $c_1$ corresponds to
\begin{equation}
\label{eq:a-upper-bound}
	a = \pi^{2/3}A^{-1/6}\alpha^{-1/3}k^{2/3} \sim k^{2/3}
\end{equation}
as $k\to \infty$, if the other parameters are fixed. For the lower bound, we again start from \eqref{eq:k-1-breakdown} but this time stretch the Dirichlet interval and use the lower bound in \eqref{eq:eig1ineq1} to obtain
\begin{equation}
\label{eq:k-1-lower-bound}
\begin{split}
	\eigmode{k}{1}{\rect{a}{A}}{\alpha} &\geq\frac{\pi^2 (k-2)^2}{Aa^2}+\frac{2\pi^2\alpha a}{A^{1/2}\left(\pi^2+\frac{2A^{1/2}\alpha}{a}\right)}\\
	&\geq \frac{\pi^2(k-2)^2}{Aa^2} + \frac{2\pi^2\alpha a}{\pi^2 A^{1/2} + 2A\alpha},
\end{split}
\end{equation}
the last inequality following since $a\geq 1$. We now make the \emph{Ansatz} $a=c_2 (k-2)^{2/3}$ to obtain
\begin{displaymath}
	\eigmodeoptrect{k}{1}{A}{\alpha} \geq \inf_{c_2\geq (k-2)^{-2/3}} \pi^2(k-2)^{2/3} \left[\frac{1}{A c_2^2} + \frac{2\alpha}{\pi^2 A^{1/2}
	+2A\alpha}c_2\right].
\end{displaymath}
Obviously, the infimum can only become smaller if we look at all $c_2>0$; in this case, we again obtain a global minimiser at a value of $c_2$ independent of $k$, namely
\begin{equation}
\label{eq:a-lower-bound}
	c_2 = \left(\frac{A^{1/2}\alpha}{\pi^2+2A^{1/2}\alpha}\right)^{-1/3},
\end{equation}
corresponding to $a=c_2(k-2)^{2/3} \sim k^{2/3}$ and resulting in the lower bound
\begin{displaymath}
	\eigmodeoptrect{k}{1}{A}{\alpha} \geq 3\pi^2\left(\pi^2 + 2A^{1/2}\alpha\right)^{-2/3}A^{-2/3}\alpha^{2/3}(k-2)^{2/3}.
\end{displaymath}
In particular, this together with a standard compactness argument establishes the existence of a minimiser for $\eigmodeoptrect{k}{1}{A}{\alpha}$ for every admissible combination of parameters. Moreover, combined with the upper bound it also shows that the long side length of the optimiser must be proportional to $k^{2/3}$.

Finally, to establish \eqref{eq:k-1-opt-asymptotics}, we refine the second inequality in \eqref{eq:k-1-lower-bound}. Namely, since we now know that the optimal side length behaves like $k^{2/3}$ and in particular tends to $\infty$ with $k$, for every $\varepsilon>0$ there exists some $k_\varepsilon \geq 1$ such that $2A\alpha/a_k^\ast < \varepsilon$ for all $k\geq k_\varepsilon$, where $a_k^\ast$ is the optimal side length value corresponding to $\eigmodeoptrect{k}{1}{A}{\alpha}$. This leads to the improved lower bound
\begin{displaymath}
	\eigmodeoptrect{k}{1}{A}{\alpha} \geq \inf_{c_2>0} \pi^2(k-2)^{2/3} \left[\frac{1}{A c_2^2} + \frac{2\alpha}{\pi^2 A^{1/2}
	+\varepsilon}c_2\right],
\end{displaymath}
provided $k\geq k_\varepsilon$ is large enough. This, in turn, leads to
\begin{displaymath}
	\frac{\eigmodeoptrect{k}{1}{A}{\alpha}}{k^{2/3}} \geq 3\pi^2\left(\pi^2+\varepsilon\right)^{-2/3}A^{-2/3}
	\alpha^{2/3}\left(\frac{k-2}{k}\right)^{2/3}
\end{displaymath}
for all $k\geq k_\varepsilon$. Letting $k\to \infty$ and then passing to the limit as $\varepsilon\to 0$ yields \eqref{eq:k-1-opt-asymptotics}, when combined with the upper bound from \eqref{eq:k-1-mode-bounds}.
\end{proof}

\begin{remark}
\label{rem:k-1-balance}
The power $k^{2/3}$ comes from balancing the effect of the first Dirichlet eigenvalue of an interval of length $\sim a/k$ with the first Robin eigenvalue of an interval of length $\sim a^{-1}$.
\end{remark}

\subsection{An estimate on the eigenvalue counting function of the Robin Laplacian on rectangles}
\label{sec:eigcount}

For given numbers $A,\alpha>0$ and $a\geq 1$, we introduce the counting function
\begin{equation}
\label{eq:eigcount-def}
	\eigcount{\rect{a}{A}}{\alpha}{\lambda} := \# \{k: \keig{k}{\rect{a}{A}}{\alpha} \leq \lambda \},
\end{equation}
for positive values of the parameter $\lambda$. We will give a simple but effective upper estimate on this function.

\begin{lemma}
\label{lem:eigcount}
Fix $\alpha>0$ and $A>0$. Then, for any $a \geq 1$,
\begin{equation}
\label{eq:eigcount-bound}
	\eigcount{\rect{a}{A}}{\alpha}{\lambda} \leq \frac{\lambda A}{\pi^2} + \frac{(\lambda A)^{1/2}}{\pi}\left(a+\frac{1}{a}\right) + 1
\end{equation}
for all $\lambda>0$.
\end{lemma}

Observe that the bound on the right-hand side of \eqref{eq:eigcount-bound} is independent of $\alpha > 0$, and indeed, for the proof, we will actually show that \eqref{eq:eigcount-bound} is an upper bound on the eigenvalue counting function of the \emph{Neumann} Laplacian. Since one is interested in maximising, not minimising, the eigenvalues of the Neumann Laplacian, previous works have correspondingly only given \emph{lower} bounds on the Neumann counting function (see, e.g., \cite{vdbbugi}). Although our bound is actually quite rough even in the Neumann case, not to mention loss in going from the Robin to the Neumann condition), it will still be sufficient to give the correct power relationship between $\alpha$ and $k$ in Theorem~\ref{thm:k-squares}.

\begin{proof}[Proof of Lemma~\ref{lem:eigcount}]
Monotonicity of the eigenvalues with respect to $\alpha \geq 0$ means that $\alpha \mapsto \eigcount{\rect{a}{A}}{\alpha}{\lambda}$ is a \emph{decreasing} function (for fixed $A$, $a$ and $\lambda$); hence, as just noted, it suffices to prove \eqref{eq:eigcount-bound} when $\alpha = 0$.

Now the eigenvalues of the Neumann Laplacian are solutions $\lambda$ of
\begin{equation}
\label{eq:lattice-count}
	\lambda = \frac{\pi^2}{Aa^2} x^2 + \frac{\pi^2}{A}a^2 y^2,
\end{equation}
where $x,y$ are nonnegative integers. In particular, $\eigcount{\rect{a}{A}}{0}{\lambda}$ gives, for fixed $\lambda$, the number of integer-valued lattice points in the first quadrant of $\R^2$ (including the $x$- and $y$-axes) lying below the curve described by \eqref{eq:lattice-count}. This number is no larger than the number of lattice points within the rectangle having the same intercepts as the curve in \eqref{eq:lattice-count}, namely $(\lambda A)^{1/2}a/\pi$ and $(\lambda A)^{1/2}/(\pi a)$, respectively. But the number of lattice points within this rectangle is certainly not more than
\begin{displaymath}
	\left(\frac{(\lambda A)^{1/2}a}{\pi}+1\right)\left(\frac{(\lambda A)^{1/2}}{\pi a}+1\right),
\end{displaymath}
which is exactly the bound in \eqref{eq:eigcount-bound}.
\end{proof}

\subsection{Completion of the proof of Theorem~\ref{thm:rectangles}}

Here we combine the previous results to prove Theorem~\ref{thm:rectangles}. Indeed, by Lemma~\ref{lem:k-1-mode-bounds}, the two-sided bounds in Theorem~\ref{thm:rectangles} are true whenever the $k^{\rm th}$ eigenvalue of an optimising rectangle for $\keigoptrect{k}{A}{\alpha}$ is given by its $(k,1)$ mode, provided that also $\alpha \leq \pi^2 A^{-1/2}k^2$ as required by Lemma~\ref{lem:k-1-mode-bounds}; the statement about the asymptotic behaviour of $\keigoptrect{k}{A}{\alpha}$ follows directly once we have our two-sided bounds.

Now we know by Lemma~\ref{lem:k-1-mode-est} that the $k^{\rm th}$ eigenvalue is always given by the $(k,1)$ mode whenever $a\geq k^{1/2}$; we thus have to consider $a\leq k^{1/2}$. To obtain the optimal power relationship between $\alpha$ and $k$, that is, that the theorem is true for a region of the form $\{ \alpha \leq Ck^{1/2}\}$, we need to divide this into two subcases: (1) $a \geq C(A,\alpha)k^{1/3}$, and (2) $a\leq C(A,\alpha)k^{1/3}$, where
\begin{displaymath}
	C(A,\alpha):=3^{1/2}\pi^{-2/3}A^{1/6}\alpha^{1/3}.
\end{displaymath}

In case (1), we simply show that the $(1,2)$ mode is always larger than the upper bound on $\eigmodeoptrect{k}{1}{A}{\alpha}$ from Lemma~\ref{lem:k-1-mode-bounds}. Indeed, we have
\begin{displaymath}
	\eigmode{1}{2}{\rect{a}{A}}{\alpha} \geq \keig{2}{\interv{A^{1/2}a^{-1}}}{\alpha} > \keig{2}{\interv{A^{1/2}a^{-1}}}{0} = \frac{\pi^2 a^2}{A}.
\end{displaymath}
Then
\begin{displaymath}
	\frac{\pi^2 a^2}{A} \geq 3\pi^{2/3}A^{-2/3}\alpha^{2/3}k^{2/3}
\end{displaymath}
provided $a\geq C(A,\alpha)k^{1/3}$. (Note that for this argument to work we do not require $a\leq k^{1/2}$; that is, it holds even if $C(A,\alpha)k^{1/3} \geq k^{1/2}$.)

In case (2), it suffices to show using the counting function \eqref{eq:eigcount-def} that any rectangle for which $a\leq C(A,\alpha)k^{1/3}$ has a higher $k^{\rm th}$ eigenvalue than the upper estimate on $\keigoptrect{k}{A}{\alpha}$ from Lemma~\ref{lem:k-1-mode-bounds}. More precisely, we wish to show that, for any $a\leq C(A,\alpha) k^{1/3}$,
\begin{displaymath}
	\eigcount{\rect{a}{A}}{\alpha}{3\pi^{2/3}A^{-2/3}\alpha^{2/3}k^{2/3}}=\eigcount{\rect{a}{A}}{\alpha}{\pi^2 A^{-1}C(A,\alpha)^2 k^{2/3}}\leq k.
\end{displaymath}
By Lemma~\ref{lem:eigcount} (more precisely, \eqref{eq:eigcount-bound}) and the fact that the function $a+a^{-1} \leq 2a$ reaches its maximum for $a \in [1,C(A,\alpha)k^{1/3}]$ at $a=C(A,\alpha)k^{1/3}$ (assuming without loss of generality that $C(A,\alpha)k^{1/3}\geq 1$, since otherwise case (1) always holds, it suffices to have
\begin{displaymath}
	3C(A,\alpha)^2 k^{2/3} + 1 \leq k.
\end{displaymath}
Using the crude bound $k-1\geq k/2$ for $k\geq 2$, this is satisfied provided
\begin{equation}
\label{eq:rectangles-rectangles-comparison}
	\alpha^2 \leq \frac{\pi^4}{18^3 A}k.
\end{equation}
Concluding, for Theorem~\ref{thm:rectangles} to hold, in addition to \eqref{eq:rectangles-rectangles-comparison} it suffices that $\alpha \leq \pi^2 A^{-1/2}k^2$; but since $k\geq 1$, this latter condition is always implied by \eqref{eq:rectangles-rectangles-comparison}. Hence we see that \eqref{eq:rectangles-rectangles-comparison} is by itself sufficient for Theorem~\ref{thm:rectangles}.

\section{Minimality of $k$ equal squares: Proof of Theorems~\ref{thm:k-squares} and~\ref{thm:brexit}}

\label{sec:uk-in-eu}

Recall that $\unionsquare{k}$ denotes the disjoint union of $k$ equal squares, of total area $A>0$. To prove Theorem~\ref{thm:k-squares}, it suffices to prove the optimality of $\unionsquare{k}$ in the claimed region. The two-sided estimate~\eqref{eq:opt-value-est} on $\keigopt{k}{A}{\alpha}$ is then simply the two-sided estimate on $\keig{k}{\unionsquare{k}}{\alpha}$ which comes from combining the lower estimate from Proposition~\ref{prop:boundkequalsquares} and the upper estimate from \eqref{eq:main-interval-bound}, cf.~also \eqref{eq:recall-square-bound}. Our proof of the optimality of $\unionsquare{k}$ uses the following strategy, firstly dealing with rectangles:
\begin{enumerate}
\item we show that the $(k,1)$ mode of any rectangle always has a larger $k^{\rm th}$ eigenvalue than $\unionsquare{k}$, meaning that whenever $a\geq k^{1/2}$
we always have $\keig{k}{\rect{a}{A}}{\alpha} > \keig{k}{\unionsquare{k}}{\alpha}$ (cf.~Lemma~\ref{lem:k-1-mode-est});
\item for $a \leq k^{1/2}$ we use the eigenvalue counting function $\eigcount{\rect{a}{A}}{\alpha}{\,\cdot\,}$ from Section~\ref{sec:eigcount} together with a simple estimate on $\keig{k}{\unionsquare{k}}{\alpha} = \keig{1}{\sq{(A/k)^{1/2}}}{\alpha}$ to show that for sufficiently large $k$ (depending on $\alpha$ in the fashion claimed in the statement of the theorem), we have $\keig{k}{\rect{a}{A}}{\alpha} > \keig{1}{\sq{(A/k)^{1/2}}}{\alpha}$;
\item to consider unions of rectangles, we proceed by induction on $k$.
\end{enumerate} 

We will give each step of the proof in a separate subsection; the proof of Theorem~\ref{thm:brexit} is then given in a further subsection at the end.

\subsection{Proof of Theorem~\ref{thm:k-squares} for long, thin rectangles via the $(k,1)$ mode}
\label{sec:uk-in-eu-thin}

We start with the following important observation.

\begin{lemma}
\label{lem:k-1-mode-vs-unionsquare}
For any $\alpha>0$, $A>0$, $k\geq 2$ and $a\geq 1$, we have
\begin{displaymath}
	\eigmode{k}{1}{\rect{a}{A}}{\alpha} \geq \keig{1}{\sq{(A/k)^{1/2}}}{\alpha} = \keig{k}{\unionsquare}{\alpha}.
\end{displaymath}
\end{lemma}

\begin{proof}
The argument is essentially the same as in Corollary~\ref{cor:lambda2}. The $(k,1)$ mode eigenvalue is equal to the first eigenvalue of either of the ``end'' nodal domains, i.e., which touch either of the shorter sides of $\rect{a}{A}$ (see Remark~\ref{rem:i-j-mode}). This, in turn, is equal to the sum of the first Robin eigenvalue of an interval of length $A^{1/2}/a$ and the first eigenvalue of a mixed Dirichlet-Robin problem on an interval of some length $A^{1/2}\hat{a}$, where $\hat{a} \leq a/k$. Replacing the Dirichlet condition by a Robin one and using Lemma~\ref{lem:homothetic} applied to the same interval, this means $\eigmode{k}{1}{\rect{a}{A}}{\alpha}$ is larger than the first Robin eigenvalue of a rectangle of side length $A^{1/2}/a$ and $A^{1/2}a/k$. Theorem~\ref{thm:lambda1} applied to this rectangle completes the proof.
\end{proof}

This immediately has the following consequence, which we summarise as a lemma for future reference.

\begin{lemma}
\label{lem:uk-in-eu-thin}
Suppose $a \geq k^{1/2}$. Then, for any $\alpha>0$, $A>0$ and $k\geq 2$, we have
\begin{displaymath}
	\keig{k}{\rect{a}{A}}{\alpha} \geq \keig{k}{\unionsquare}{\alpha}.
\end{displaymath}
\end{lemma}

\begin{proof}
Combine Lemma~\ref{lem:k-1-mode-vs-unionsquare} and Lemma~\ref{lem:k-1-mode-est}.
\end{proof}

\subsection{Proof of Theorem~\ref{thm:k-squares} for relatively fat rectangles via the eigenvalue counting function}
\label{sec:uk-in-eu-fat}

Here, we wish to consider $\rect{a}{A}$ for $a \leq k^{1/2}$. We start by recalling the following upper bound on $\keig{k}{\unionsquare{k}}{\alpha} = \keig{1}{\sq{(A/k)^{1/2}}}{\alpha}$ from \eqref{eq:main-interval-bound}:
\begin{equation}
\label{eq:recall-square-bound}
	\keig{1}{\sq{(A/k)^{1/2}}}{\alpha} = 2\keig{1}{\interv{(A/k)^{1/2}}}{\alpha} < \frac{4k^{1/2}\alpha}{A^{1/2}}.
\end{equation}

\begin{lemma}
\label{lem:square-fat}
Suppose, given $A>0$ and $\alpha>0$, that $k \geq 3$ is such that
\begin{equation}
\label{eq:ugly-bound}
	\frac{4A^{1/2}k^{1/2}\alpha}{\pi^2}+\frac{2A^{1/4}k^{1/4}\alpha^{1/2}}{\pi}
	\left(k^{1/2} + k^{-{1/2}}\right) + 1 \leq k.
\end{equation}
Then, for any $a \leq k^{1/2}$, we have have $\keig{k}{\rect{a}{A}}{\alpha} \geq \keig{k}{\unionsquare{k}}{\alpha}$.
\end{lemma}

\begin{proof}
The proof uses the eigenvalue counting function $\eigcount{\Omega}{\alpha}{\lambda}$ defined in Section~\ref{sec:eigcount}. By \eqref{eq:recall-square-bound}, we know that $\eigcount{\unionsquare{k}}{\alpha}{4k^{1/2}\alpha/A^{1/2}} \geq k$; hence, to prove the lemma, it suffices to show that, assuming \eqref{eq:ugly-bound} and $a \leq k^{1/2}$,
\begin{displaymath}
	\eigcount{\rect{a}{A}}{\alpha}{4k^{1/2}\alpha/A^{1/2}} \leq k.
\end{displaymath}
To show this, we first observe that the bound \eqref{eq:eigcount-bound} from Lemma~\ref{lem:eigcount} is a monotonically increasing function of $a \geq 1$, so suffices to consider the extremal case $a=k^{1/2}$ in \eqref{eq:eigcount-bound}, that is,
\begin{displaymath}
	\eigcount{\rect{a}{A}}{\alpha}{4k^{1/2}\alpha/A^{1/2}} 
	\leq\frac{4A^{1/2}k^{1/2}\alpha}{\pi^2}+\frac{2A^{1/4}k^{1/4}\alpha^{1/2}}{\pi}
	\left(k^{1/2} + k^{-{1/2}}\right) + 1.
\end{displaymath}
Thus \eqref{eq:ugly-bound} guarantees that $\eigcount{\rect{a}{A}}{\alpha}{4k^{1/2}\alpha/A^{1/2}} \leq k$.
\end{proof}

Before proceeding, let us give a somewhat weaker but considerably simpler alternative to \eqref{eq:ugly-bound}, which still gives the correct power dependence between $\alpha$ and $k$. 
If we posit a relationship of the form $\alpha = Ck^{1/2}$ in \eqref{eq:ugly-bound}, we obtain
\begin{equation}
\label{eq:not-so-ugly-bound}
	\frac{4A^{1/2}k}{\pi^2}C + \frac{(2k+2)A^{1/4}}{\pi}C^{1/2} + 1 \leq k.
\end{equation}
The corresponding equality is a quadratic equation in $A^{1/4}C^{1/2}/\pi$ with a unique positive solution
\begin{displaymath}
	\frac{A^{1/4}C^{1/2}}{\pi} = \frac{-k-1+\sqrt{5k^2-2k+1}}{4k}
\end{displaymath}
below which \eqref{eq:not-so-ugly-bound} and thus \eqref{eq:ugly-bound} hold. An elementary but tedious calculation shows that this solution is monotonically increasing in $k\geq 3$; thus, to obtain a universally valid bound, it suffices to take $k=3$ (the case $k=2$ being covered by Corollary~\ref{cor:lambda2}). In this case \eqref{eq:not-so-ugly-bound} reduces to
\begin{displaymath}
	\frac{12}{\pi^2}A^{1/2}C + \frac{8}{\pi}A^{1/4}C^{1/2} - 2 \leq 0.
\end{displaymath}
The largest possible value of $C$ for which this holds is
\begin{displaymath}
	C = \left(\frac{\sqrt{10}-2}{6}\right)^2 \pi^2 A^{-1/2} = \frac{\pi^2}{18}(7-2\sqrt{10})A^{-1/2}.
\end{displaymath}
Thus, using the fact that if \eqref{eq:ugly-bound} holds for some $\alpha_0>0$, then it holds for all $\alpha \in (0,\alpha_0)$, we see that to satisfy \eqref{eq:ugly-bound} and thus obtain Theorem~\ref{thm:k-squares} for any rectangle $\rect{a}{A}$ with $a \leq k^{1/2}$ it is certainly sufficient that
\begin{equation}
\label{eq:less-ugly-bound}
	\alpha \leq Ck^\frac{1}{2} = \frac{\pi^2}{18}(7-2\sqrt{10}) A^{-1/2}k^{1/2} \approx 0.370 A^{-1/2}k^{1/2}.
\end{equation}

We finish this subsection by summarising how the above steps complete the proof of Theorem~\ref{thm:k-squares} for rectangles.

\begin{proof}[Proof of Theorem~\ref{thm:k-squares} for rectangles]
Fix $\alpha>0$. Choose $k^\ast$ to be the smallest $k\geq 3$ such that \eqref{eq:ugly-bound} (or \eqref{eq:less-ugly-bound}) holds for this $k^\ast = k^\ast (A,\alpha)$ (and hence also for all $k\geq k^\ast$). Now fix $a \geq 1$ and $k \geq k^\ast$. If $a \leq k^\frac{1}{2}$, then by Lemma~\ref{lem:square-fat}, we have $\keig{k}{\rect{a}{1}}{\alpha} > \keig{k}{\unionsquare{k}}{\alpha}$. If $a \geq k^\frac{1}{2}$, then we may apply Lemma~\ref{lem:uk-in-eu-thin}.

Finally, note that \eqref{eq:ugly-bound} gives an explicit estimate on $k^\ast$. Indeed, reformulating \eqref{eq:less-ugly-bound} (which still exhibits the asymptotically correct power relationship between $k^\ast$, $A$ and $\alpha$),
\begin{displaymath}
	k^\ast \geq \left(\frac{18}{7-2\sqrt{10}}\right)^2\frac{A}{\pi^4} \alpha^2 \approx 7.291 A \alpha^2
\end{displaymath}
is sufficient.
\end{proof}

\subsection{Disjoint unions of rectangles}
\label{sec:disjoint-unions}

We start by formulating an abstract result on minimisers of disjoint unions of domains. We change perspective slightly: instead of fixing $\alpha$ and showing that a certain type of domain minimises the $k^{\rm th}$ eigenvalue for $k$ large enough, it will be more useful to fix $k$ and consider the corresponding range of $\alpha$ small enough. Here we do not restrict ourselves to rectangles: we will work with a general admissible family in the sense of Definition~\ref{def:admissible-family}.

\begin{lemma}
\label{lem:disjoint-growth}
Fix $A>0$ and suppose $\family$ is an admissible family of planar domains for $A$, such that for all $k\geq 1$ and $\alpha >0$ there is a domain in $\family$ realising
\begin{displaymath}
	\inf \{ \keig{k}{\Omega}{\alpha}: \Omega \in \family\}.
\end{displaymath}
Suppose also that, when $k=1$, the minimiser $\Omega^\ast \in \family$ is independent of $\alpha>0$. Assume in addition that there exists a sequence of numbers
\begin{displaymath}
	0 < \alpha_{2} \leq \alpha_{3} \leq \alpha_{4} \leq \ldots
\end{displaymath}
such that, for any $k \geq 2$,
\begin{displaymath}
	\inf \{ \keig{k}{\Omega}{\alpha} : \Omega \in \family \text{ is connected}\}
\end{displaymath}
is realised by the (non-connected) domain $\Omega_k^\ast$ consisting of $k$ equal copies of $\frac{1}{\sqrt{k}}\Omega^\ast$ for all $\alpha \in (0,\alpha_k]$. Define
\begin{equation}
\label{eq:disjoint-transfer}
	\alpha_k^\ast := \min \left\{ \alpha_k, \sqrt{\frac{k}{k-1}}\alpha_{k-1},\ldots,\sqrt{\frac{k}{2}}\alpha_{2} \right\}
\end{equation}
for $k\geq 2$. Then $\Omega_k^\ast$ also realises
\begin{displaymath}
	\inf \{ \keig{k}{\Omega}{\alpha} : \Omega \in \family\}
\end{displaymath}
for all $\alpha \in (0,\alpha_k^\ast]$. If $\alpha_k \to \infty$, then also $\alpha_k^\ast \to \infty$ as $k\to\infty$.
\end{lemma}

Thus \eqref{eq:disjoint-transfer} shows how we can go from minimisers among connected domains to minimisers among disjoint unions of domains.

\begin{proof}
1. We first recall the scaling relation for the optimal values from Lemma~\ref{lem:optimal-scaling}(2): suppose that for some $k\geq 2$ the value of $\alpha_k$, resp.~$\alpha_k^\ast$, is given and corresponds to some domain $\widetilde\Omega$. If $B<A$, then the corresponding values among the scaled-down family $\{\frac{1}{\sqrt{B}}\Omega: \Omega \in \family \}$ are $\alpha_k\sqrt{A/B} > \alpha_k$ and $\alpha_k^\ast\sqrt{A/B} > \alpha_k^\ast$, respectively.

2. We now proceed by induction on $k$. For $k=2$, the statement recalls closely part 2 of the proof of Corollary~\ref{cor:lambda2}: if $\Omega \in \family$ is connected, then $\keig{2}{\Omega}{\alpha} \geq \keig{2}{\Omega_k^\ast}{\alpha}$ for all $\alpha\leq \alpha_2$. If $\Omega \in \family$ is not connected, then either it has one connected component whose second eigenvalue is $\keig{2}{\Omega}{\alpha}$. Discarding the rest of $\Omega$ and inflating this component decreases the second eigenvalue, which is in particular larger than $\keig{2}{\Omega_k^\ast}{\alpha}$ for all $\alpha\leq \alpha_2$. If $\keig{2}{\Omega}{\alpha} = \keig{1}{\Omega''}{\alpha}$ for some connected component $\Omega''$ of $\Omega$, where another connected component $\Omega'$ gives $\keig{1}{\Omega}{\alpha}$, then we may replace $\Omega'$ and $\Omega''$ by scaled copies of $\Omega^\ast$. The second eigenvalue of their union is either (depending on the ratios of their areas) always larger than $\keig{2}{\Omega_2^\ast}{\alpha}$ (if the two areas are roughly equal) or, if one is much larger than the other and hence the second eigenvalue of the union equals the second eigenvalue of one copy, then, using Step 1, it is at least always larger than $\keig{2}{\Omega_2^\ast}{\alpha}$ for $\alpha \leq \alpha_2$. Hence we obtain the conclusion for $\alpha_2^\ast =\alpha_2$.

3. We now give the induction step. Suppose the lemma is true for $\alpha_2^\ast,\ldots,\alpha_{k-1}^\ast$ and consider $\alpha_{k}^\ast$. Obviously, for $\alpha \leq \alpha_k$ the minimiser, assumed to exist, cannot be connected. Fix such an $\alpha \leq \alpha_k$. By the Wolf--Keller principle (see Lemma~\ref{lem:wolf-keller}),
if $\Omega_\alpha$ is the minimiser, then
\begin{displaymath}
	\Omega_\alpha = t_i \Omega_i \cup t_{k-i} \Omega_{k-i},
\end{displaymath}
where $\Omega_i$ and $\Omega_{k-i}$ are the minimisers of $\keig{i}{\,\cdot\,}{\alpha_i}$ and $\keig{k-i}{\,\cdot,\,}{\alpha_{k-i}}$ for some $\alpha_i, \alpha_{k-i}$ related to $\alpha$, respectively, and where $t_i^2 + t_{k-i}^2 = 1$. Moreover, the scaling factors $t_i$ and $t_{k-i}$ are chosen such that
\begin{displaymath}
	\keig{k}{\Omega_\alpha}{\alpha} = \keig{i}{t_i\Omega_i}{\alpha} = \keig{k-i}{t_{k-i}\Omega_{k-i}}{\alpha}.
\end{displaymath}
In particular, either $t_i^2 \leq i/k$ or $t_{k-i}^2 \leq (k-i)/k$. In the first case, by the induction hypothesis and Step 1,
\begin{displaymath}
	\keig{i}{t_i\Omega_i}{\alpha} \geq \keig{i}{t_i\Omega_i^\ast}{\alpha}\qquad \text{if } \alpha \leq \frac{\alpha_i^\ast}{t_i} 
	\leq \sqrt{\frac{k}{i}}\alpha_i^\ast,
\end{displaymath}
using that $t_i^2 \leq i/k$. Moreover, in this case, $t_i\Omega_i^\ast$ consists of $i$ equal copies of $\Omega^\ast$ each of area
\begin{displaymath}
	t_i^2 \leq \frac{i}{k}\cdot\frac{1}{i} = \frac{1}{k}.
\end{displaymath}
Thus, for $\alpha \leq \alpha_i^\ast\sqrt{k/i}\alpha_i^\ast$, we have
\begin{displaymath}
	\keig{k}{\Omega_\alpha}{\alpha} \geq \keig{i}{t_i\Omega_i^\ast}{\alpha} \geq \keig{k}{\Omega_k^\ast}{\alpha},
\end{displaymath}
and thus $\Omega_k^\ast$ is a minimiser in this case. Similarly, if $t_{k-i}^2 \leq (k-i)/k$, then we obtain the optimality of $\Omega_k^\ast$ whenever $\alpha \leq \alpha_{k-i}^\ast \sqrt{k/(k-i)}$. Concluding,
\begin{displaymath}
	\keig{k}{\Omega_\alpha}{\alpha} \geq \keig{k}{\Omega_k^\ast}{\alpha} \qquad \text{if } \alpha \leq \min \left\{ \sqrt{\frac{k}{i}}\alpha_i^\ast,
	\sqrt{\frac{k}{k-i}}\alpha_{k-i}^\ast \right\}.
\end{displaymath}
Repeating this argument over all possible pairs $(i,k-i)$, $i=1,\ldots$, we obtain that $\Omega_k^\ast$ is minimal for all
\begin{displaymath}
	\alpha \leq \hat\alpha_{k}:= \min \left\{ \alpha_{k}, \sqrt{\frac{k}{k-1}}\alpha_{k-1}^\ast,\ldots,\sqrt{\frac{k}{2}}\alpha_{2}^\ast \right\}.
\end{displaymath}

4. Finally, another simple induction argument shows that $\hat \alpha_k$ is equal to $\alpha_k^\ast$ given by \eqref{eq:disjoint-transfer}. Indeed, for $\alpha_3^\ast$, since $\alpha_2=\alpha_2^\ast$, we have
\begin{displaymath}
	\hat \alpha_3 = \min \left\{ \alpha_3, \sqrt{\frac{3}{2}}\alpha_2^\ast \right\} = \min \left\{ \alpha_3, \sqrt{\frac{3}{2}}\alpha_2 \right\}
	=\alpha_3^\ast.
\end{displaymath}
Similarly, if $\alpha_{i}^\ast = \hat \alpha_i$ for all $i=1,\ldots,k-1$, then
\begin{displaymath}
\begin{aligned}
	\hat \alpha_k &= \min \left\{ \alpha_k, \sqrt{\frac{k}{k-1}}\min\left\{ \alpha_{k-1},\sqrt{\frac{k-1}{k-2}}\alpha_{k-2},\ldots,\sqrt{\frac{k-1}{2}}
	\alpha_2 \right\},\ldots, \sqrt{\frac{k}{2}}\alpha_2 \right\}\\
	&= \min \left\{ \alpha_k, \sqrt{\frac{k}{k-1}}\alpha_{k-1},\ldots, \sqrt{\frac{k}{2}}\alpha_2 \right\} = \alpha_k^\ast.
\end{aligned}
\end{displaymath}
We conclude that $\keig{k}{\,\cdot\,}{\alpha}$ is minimsed by $\Omega_k^\ast$ whenever $\alpha \leq \hat\alpha_k = \alpha_k^\ast$.

5. The statement that $\alpha_k \to \infty$ implies $\alpha_k^\ast \to \infty$ is elementary and follows, for example, from a simple contradiction argument.
\end{proof}

The formula given by \eqref{eq:disjoint-transfer} becomes particularly simple if the optimal value $\alpha_k$ for connected domains from Lemma~\ref{lem:disjoint-growth} behaves like $\sqrt{k}$ (as is the case for our rectangles and as generally appears to be the case for the Robin problem in two dimensions).

\begin{lemma}
\label{lem:precise-disjoint-growth}
Keep the notation and assumptions from Lemma~\ref{lem:disjoint-growth}. If, in addition, there exists a constant $C=C(\family)>0$ such that
\begin{equation}
\label{eq:disjoint-choice}
	\alpha_k = C \sqrt{k}
\end{equation}
for all $k \geq 2$, then $\alpha_k^\ast = \alpha_k = C\sqrt{k}$ for all $k\geq 2$.
\end{lemma}

\begin{proof}
Inserting \eqref{eq:disjoint-choice} into \eqref{eq:disjoint-transfer}, since
\begin{displaymath}
	\sqrt{\frac{k}{k-j}}\alpha_{k-j} = C\sqrt{\frac{k}{k-j}}\sqrt{k-j}=C\sqrt{k}=\alpha_k
\end{displaymath}
for all $j=1,\dots,k-1$, we immediately obtain $\alpha_k^\ast = \alpha_k$.
\end{proof}

With this preparation, we can now treat the case of disjoint unions of rectangles, that is, complete the proof of Theorem~\ref{thm:k-squares}.

\begin{proof}[Proof of Theorem~\ref{thm:k-squares} for disjoint unions of rectangles]
We already proved at the end of Section~\ref{sec:uk-in-eu-fat} that if $k\geq 2$ and
\begin{displaymath}
	\alpha_k := \underbrace{\frac{\pi^2}{18}(7-2\sqrt{10})}_{=:C} A^{-1/2} k^{1/2}
\end{displaymath}
then for any rectangle $\Omega$ of area $A$, $\keig{k}{\Omega}{\alpha}$ is no smaller than $\keig{k}{\unionsquare{k}}{\alpha}$ whenever $\alpha \in (0,\alpha_k]$. Thus, by Lemma~\ref{lem:precise-disjoint-growth}, the same is true for all disjoint unions of rectangles whenever $\alpha \leq \alpha_k^\ast = \alpha_k = C\sqrt{k}$.
\end{proof}

\subsection{Non-optimality of $\unionsquare{k}$ for large $\alpha$: Transition between unions of squares and the proof of Theorem~\ref{thm:brexit}}
\label{sec:transition}
The domain $\unionsquare{k}$ stops being optimal at the latest at the point where $k$ equal squares of area $A/k$ have the same $k^{\rm th}$ eigenvalue as the domain consisting of $k-3$ equal squares and one larger square with area three times that of the other smaller squares.
The curve where this happens is defined by the following identity
\[
 \keig{1}{\sq{\sqrt{A/k}}}{\alpha} = \keig{2}{\sq{\sqrt{3A/k}}}{\alpha} = \keig{3}{\sq{\sqrt{3A/k}}}{\alpha}.
\]
Writing $x_{1} = \sqrt{\keig{1}{\interv{\sqrt{A/k}}}{\alpha}}$, $x_{2} = \sqrt{\keig{1}{\interv{\sqrt{3A/k}}}{\alpha}}$ and
$x_{3} = \sqrt{\keig{2}{\interv{\sqrt{3A/k}}}{\alpha}}$ this is equivalent to the following system of equations
\begin{equation}\label{transition}
 \left\{
 \begin{array}{lll}
  \alpha & = & x_{1}\tan\left( \fr{\sqrt{A} x_{1}}{2\sqrt{k}}\right)\eqskip
  \alpha & = & x_{2}\tan\left( \fr{\sqrt{3A}x_{2}}{2\sqrt{k}}\right)\eqskip
  \alpha & = & -x_{3}\cot\left( \fr{\sqrt{3A}x_{3}}{2\sqrt{k}}\right)\eqskip
  2x_{1}^2 & = & x_{2}^{2}+x_{3}^2
 \end{array}.
 \right.
\end{equation}
Based on this, we shall now prove a result regarding the existence of such a transition curve.
\begin{theorem}\label{transition3to1}
 There exists a solution of system~\eqref{transition} of the form
 \[
  \alpha = \fr{C}{\sqrt{A}} \sqrt{k},
 \]
 where the constant $C$ satisfies $4/5< C < 5\pi^2/2$.
\end{theorem}
\begin{remark}
\label{rem:transition3to1}
 Numerically, we obtain that the solution of system~\eqref{transition} which yields the lowest positive value of the constant $C$ is
 \[
  (x_{1},x_{2},x_{3})\approx(2.50386,1.57707,3.1704)\sqrt{\fr{k}{A}},
 \]
corresponding to
 \[
  \alpha \approx 7.58442 \sqrt{\fr{k}{A}}.
 \]
\end{remark}
\begin{proof}
 We look for solutions of system~\eqref{transition} of the form $\alpha = C\sqrt{k}$ and $x_{i} = 2c_{i}\sqrt{k/A}$, $(i=1,2,3)$. Replacing this
 in~\eqref{transition} yields the system
 \[
 \left\{
 \begin{array}{lll}
  C & = & \fr{2c_{1}}{\sqrt{A}}\tan\left( c_{1} \right)\eqskip
  C & = & \fr{2c_{2}}{\sqrt{A}}\tan\left( \sqrt{3}c_{2}\right)\eqskip
  C & = & -\fr{2c_{3}}{\sqrt{A}}\cot\left( \sqrt{3}c_{3}\right)\eqskip
  2c_{1}^2 & = & c_{2}^{2}+c_{3}^2
 \end{array}.
 \right.
 \]
We now eliminate $C$ from the equations by equating the left-hand sides of the first equation to those of the second and third equations. This
yields the new system in $c_{1}, c_{2}$ and $c_{3}$
\begin{equation}\label{transition2}
\left\{
 \begin{array}{lll}
  c_{1} \tan\left( c_{1} \right) = c_{2}\tan\left( \sqrt{3}c_{2}\right)\eqskip
  c_{1} \tan\left( c_{1} \right) = -c_{3}\cot\left( \sqrt{3}c_{3}\right)\eqskip
  2c_{1}^{2} = c_{2}^2+ c_{3}^{2} 
 \end{array}.
\right.
\end{equation}
We shall now prove the existence of a solution of the above system with smallest possible $c_{1}$, that is, for $c_{1}$ on the interval $(0,\pi/2)$. Note
that since $C= c_{1}\tan\left(\sqrt{A} c_{1}/2\right)$ is increasing in $c_{1}$, this yields the smallest possible value for $C$ for a given value
of the area. The solutions to other constants $c_{2}$ and $c_{3}$ belong to the intervals $\left(0,\pi/(2\sqrt{3})\right)$ and $\left(\pi/(2\sqrt{3}),\pi/\sqrt{3}\right)$,
respectively.

We first note that since the function $x \mapsto x\tan(a x)$ is increasing in $x$ for $x\in(0,\pi/2)$, the first equation in~\eqref{transition2} defines $c_{2}$
as a continuous increasing function of $c_{1}$ defined on $[0,\pi/2)$ and with values in $\left[0,\pi/(2\sqrt{3})\right)$. This function (which abusing
notation we denote by $c_{2}(c_{1})$), satisfies
\[
 c_{2}(0) = 0 \mbox{ and } \lim_{c_{1}\to(\pi/2)^{-}} c_{2}(c_{1}) = \fr{\pi}{2\sqrt{3}}. 
\]
Similarly, the second
equation in~\eqref{transition2} defines $c_{3}$ as a continuous increasing function of $c_{1}$ on the interval $[0,\pi/2)$ and with values on
$\left[\pi/(2\sqrt{3}),\pi/\sqrt{3}\right)$. This second function satisfies
\[
 c_{3}(0) = \pi \mbox{ and } \lim_{c_{1}\to(\pi/2)^{-}} c_{3}(c_{1}) = \fr{\pi}{\sqrt{3}}. 
\]
Defining now the function $F$ on the interval $[0,\pi)$ by
\[
 F(c_{1}) = 2c_{1}^2-c_{2}^{2}\left(c_{1}\right)-c_{3}^{2}\left(c_{1}\right),
\]
we see that this is continuous and satisfies
\[
 F(0) = - \fr{\pi^2}{12} \mbox{ and } \lim_{c_{1}\to(\pi/2)^{-}} F(c_{1}) = 2\times\fr{\pi^2}{4} - \fr{\pi^2}{12} -\fr{\pi^2}{3} = \fr{\pi^2}{12}.
\]
Hence $F$ must vanish somewhere on the interval $(0,\pi/2)$, implying the existence of at least one solution of system~\eqref{transition2} (and
hence~\eqref{transition}).

From the third equation in~\eqref{transition2} and the fact that $c_{2}$ and $c_{3}$ lie on the intervals $\left[0,\pi/(2\sqrt{3})\right)$ and $\left[\pi/(2\sqrt{3}),\pi/\sqrt{3}\right)$,
respectively, we have
\[
 \fr{\pi}{2\sqrt{6}} < c_{1} < \fr{\sqrt{5}\pi}{2\sqrt{6}}.
\]
Since $C = \fr{2c_{1}}{\sqrt{A}}\tan\left(c_{1}\right)$, and using the (monotone) bounds for the tangent given by~\eqref{tanbounds} and
the above bounds for $c_{1}$ we obtain
\[
 \fr{4}{5} < 2c_{1}\tan(c_{1})< \fr{5\pi^2}{2}
\]
yielding the desired estimates for the constant $C$.
\end{proof}

\section{Sums of eigenvalues: Proof of Corollary~\ref{cor:sumeigopt}}
\label{sec:sums}

Given fixed positive numbers $A$ and $\alpha$, we shall now consider the smallest value attainable by the sum of the first $k$ eigenvalues $\sumeigopt{k}{A}{\alpha}$ defined in \eqref{eq:sumeigopt}. Before we proceed with the proof of Corollary~\ref{cor:sumeigopt}, we give a couple of remarks.

\begin{remark}
(a) An easy argument similar to that of Theorem~\ref{thm:existence} shows that the infimum in \eqref{eq:sumeigopt} is always attained. We leave it as an open problem actually to determine the domains which realise $\sumeigopt{k}{A}{\alpha}$, although Corollary~\ref{cor:sumeigopt}, together with Theorem~\ref{thm:rectangles}, strongly suggest that the number of connected components of the optimiser should grow with $k$, and indeed it seems natural to expect that $\unionsquare{k}$ should be the minimising domain for $\sumeigopt{k}{A}{\alpha}$ for $k$ sufficiently large.

(b) Corollary~\ref{cor:sumeigopt} also suggests that the optimal sum of eigenvalues taken over \emph{all} planar domains of area $A$, not just unions of rectangles, should also grow like $C A^{-1/2}\alpha k^{3/2}$ for some $0<C\leq 2\pi^{1/2}$ (the corresponding value for $k$ balls).
\end{remark}

Let us formulate these claims explicitly as a conjecture.
\begin{conjecture}
Fix positive numbers $A$ and $\alpha$. Then, for $k$ sufficiently large,
\begin{equation}
\label{eq:general-sum}
	\inf \left\{ \sum_{j=1}^k \keig{k}{\Omega}{\alpha}: \Omega \subset \R^2 \text{ Lipschitz, } |\Omega|=A \right\}
\end{equation}
is achieved by the disjoint union of $k$ equal disks, of total area $A$. In particular, \eqref{eq:general-sum} behaves asymptotically like
\begin{displaymath}
	\frac{2\pi^{1/2}\alpha}{A^{1/2}} k^{3/2}
\end{displaymath}
as $k\to\infty$.
\end{conjecture}

\begin{proof}[Proof of Corollary~\ref{cor:sumeigopt}]
For the upper bound, for each $k\geq 1$ we use the disjoint union $\unionsquare{k}$ of $k$ equal squares as a test domain:
\begin{displaymath}
	\sumeigopt{k}{A}{\alpha} \leq \sum_{j=1}^k \keig{j}{\unionsquare{k}}{\alpha} = k\keig{k}{\unionsquare{k}}{\alpha}
	\leq k \cdot \frac{4k^{1/2}\alpha}{A^{1/2}},
\end{displaymath}
the latter inequality following as usual from \eqref{eq:main-interval-bound}.

The lower bound follows from Theorem~\ref{thm:existence} in the form of Corollary~\ref{cor:optimal-asymptotic}. We start by observing that
\begin{displaymath}
	\sumeigopt{k}{A}{\alpha} \geq \sum_{j=1}^k \keigopt{k}{A}{\alpha}.
\end{displaymath}
We will use the asymptotics in Corollary~\ref{cor:optimal-asymptotic} to control the latter sum. Indeed, by this corollary, for fixed $A>0$ and $\alpha>0$, there exists a constant $m_1>0$ such that
\begin{displaymath}
	\keigopt{k}{A}{\alpha} \geq \frac{4\alpha}{A^{1/2}}k^{1/2} - m_1
\end{displaymath}
for all $k\geq 1$ (use the fact that $\keigopt{k}{A}{\alpha} = \frac{4\alpha}{A^{1/2}}k^{1/2} + \bo(1)$ as $k\to\infty$, by \eqref{eq:optimal-asymptotic}). Hence
\begin{displaymath}
	\sumeigopt{k}{A}{\alpha} \geq \frac{4\alpha}{A^{1/2}} \sum_{j=1}^k j^{1/2} - m_1 k.
\end{displaymath}
Now
\begin{displaymath}
	\sum_{j=1}^k j^{1/2} = \frac{2}{3}k^{3/2} + \frac{1}{2}k^{1/2} + \bo(1)
	\geq \frac{2}{3}k^{3/2} + \frac{1}{2}k^{1/2} - m_2
\end{displaymath}
for some constant $m_2>0$ independent of $k\geq 1$. Hence
\begin{displaymath}
	\sumeigopt{k}{A}{\alpha} \geq \frac{4\alpha}{A^{1/2}}\left(\frac{2}{3}k^{3/2} + \frac{1}{2}k^{1/2} - m_2\right)- m_1 k
\end{displaymath}
for all $k\geq 1$. Dividing by $k^{3/2}$ and passing to the limit yields the lower bound.
\end{proof}

\section{The higher-dimensional case}
\label{sec:higher-dimension}

To keep both the notation and the arguments as simple as possible, we have restricted ourselves to the planar case; nevertheless, we expect analogous statements to hold in $d\geq 3$ dimensions, where in place of rectangles one considers \emph{hyperrectangles} (sometimes also called \emph{cuboids} or \emph{rectangular parallelepipeds}) and their disjoint unions. Moreover, in most cases the proofs should be directly adaptable. We give a brief summary.

(1) The existence of a domain minimising $\keig{k}{\Omega}{\alpha}$ among all $d$-dimensional hyperrectangles (and among all disjoint unions of hyperrectangles, respectively) of given total volume follows from the same blow-up and continuity argument as in Theorem~\ref{thm:existence}.

(2) The minimiser of $\keig{1}{\Omega}{\alpha}$ should be the regular hypercube. However, the computation given in Theorem~\ref{thm:lambda1} will not work as easily. Once one has the hypercube for the first eigenvalue, the same proof as the one of Corollary~\ref{cor:lambda2} (noting that the nodal domains of any second eigenvalue on a hyperrectangle are again hyperrectangles) implies that the second eigenvalue is, as usual, minimised by the disjoint union of two equal regular hypercubes.

(3) The statements of Theorems~\ref{thm:k-squares} and~\ref{thm:brexit} should still hold (when dimensionally adjusted). Moreover, the proof schemes should still work, although Steps 1 and 2 of Section~\ref{sec:uk-in-eu} are more complicated due to the greater number of possible ways and directions in which a $d$-dimensional hyperrectangle can become unbounded. If $\keigopt{k}{V}{\alpha}$ now denotes the minimal $k^{\rm th}$ eigenvalue among all unions of hyperrectangles of volume $V$ in $d$ dimensions, then the correct power growth will be $k^{1/d}$ and we should have
\begin{equation}
\label{eq:optimal-asymptotic-d}
	\lim_{k\to\infty} \frac{\keigopt{k}{V}{\alpha}}{k^{1/d}} = \frac{2d\alpha}{V^{1/d}}
\end{equation}
corresponding to the $k^{\rm th}$ eigenvalue of the disjoint union $\unioncube{k}$ of $k$ equal hypercubes, each of volume $V/k$. This in turn equals first eigenvalue of a $d$-dimensional regular hypercube of volume $V/k$ and thus side length $(V/k)^{1/d}$. 

(4) Moreover, $\unioncube{k}$ should be optimal in a region of the form $\alpha \leq C k^{1/d}$, at the point where the first eigenvalue of a cube of side length $k^{-1/d}$ is equal to the $(d+1)$-st eigenvalue of a cube of side length $((d+1)k)^{-1/d}$; the analogous argument for balls was already given in \cite[Lemma~4.1]{anfrke}. Additionally, if $\Omega$ is any \emph{fixed} hyperrectangle (or union thereof), then we claim that $\unioncube{k}$ is also better than $\Omega$ in a region of the form $\alpha \leq C_\Omega k^{1/d}$, which suggests that the region of transition from $k$ equal cubes to a connected optimiser is generally quite ``thin''. To lend weight to this assertion, we make use of the counting function of $\Omega$ (cf.~\eqref{eq:eigcount-def}). At energy $\lambda>0$ we have
\begin{equation}
\label{eq:omega-count}
	\eigcount{\Omega}{\alpha}{\lambda} \sim \lambda^{d/2}.
\end{equation}
Since $\keig{k}{\unioncube{k}}{\alpha} \leq 2d\alpha/k^{1/d}$ (as follows, e.g., from \eqref{eq:eig1expa} or the bounds on $\keig{1}{\interv{a}}{\alpha}$ in Appendix~\ref{sec:interval}, choosing $a=k^{-1/d}$), arguing as in the proof of Lemma~\ref{eq:eigcount-bound}, for $\keig{k}{\unioncube{k}}{\alpha}$ to be smaller we want
\begin{displaymath}
	\eigcount{\Omega}{\alpha}{2d\alpha/k^{1/d}} \leq k.
\end{displaymath}
Using \eqref{eq:omega-count}, this is equivalent to $\alpha \leq C k^{1/d}$. In fact, this argument can easily be made into a rigorous proof; equally, with a lower bound on the counting function a similar argument could be used to show that $\alpha \geq \widetilde{C}(\Omega) k^{1/d}$ implies the fixed domain $\Omega$ is better than $\unioncube{k}$. We will consider this question in more detail below.

(5) We \emph{expect} the optimal hyperrectangle for $\keigoptrect{k}{V}{\alpha}$ to be long in one direction and short in the remaining $d-1$ (in fact, it should be the cross product of a small $d-1$-dimensional regular hypercube with a long interval). An argument similar to the one of Lemma~\ref{lem:k-1-mode-bounds} (cf.~also Remark~\ref{rem:k-1-balance}), with the \emph{Ansatz} $a=ck^\gamma$ (and short sides thus each proportional to $k^{-\gamma/(d-1)}$) leads to the power condition $2-2\gamma=\gamma/(d-1)$, i.e., $\gamma=(2d-2)/(2d-1)$ and thus to the conjecture
\begin{displaymath}
	\keigoptrect{k}{V}{\alpha} \sim k^\frac{2}{2d-1}
\end{displaymath}
as $k\to\infty$, corresponding to a long side of length proportional to $k^{(2d-2)/(2d-1)}$ and $d-1$ short sides of length like $k^{-2/(2d-1)}$. In particular, in any dimension we expect deviation (even among convex domains) from the power coming from the Weyl asymptotics for any given domain, namely $\keig{k}{\Omega}{\alpha} \sim k^{2/d}$.

(6) Based on \eqref{eq:optimal-asymptotic-d}, the smallest possible value $\sumeigopt{k}{V}{\alpha}$ of the sum of the first $k$ eigenvalues of a union of hyperrectangles with total volume $V$ should grow like $k^{1+1/d}$ as $k\to\infty$.

We will now give some more detailed considerations about the regions where we may expect the disjoint union of $k$ equal hypercubes to be the extremal domain, and where this will no
longer be the case. Let $\Omega$ be a given finite disjoint union of hyperrectangles with volume $V$ and let $\unioncube{k}$ denote the disjoint union of $k$ equal hypercubes, also of total volume $V$. We then have
\begin{equation}\label{lambdakD}
 \keig{k}{\Omega}{\alpha} < \keigd{k}{\Omega} = \fr{4\pi^2}{\left(V\omega_{d}\right)^{2/d}} k^{2/d} + r_{1}(k),
\end{equation}
where the remainder term satisfies $r_{1}(k) = \so\left(k^{2/d}\right), \mbox{ as } k\to +\infty$, and is independent of $\alpha$.
On the other hand, we also have
\[
\begin{array}{lll}
  \keig{k}{\unioncube{k}}{\alpha} & = & \keig{1}{\left(V k^{-1}\right)^{1/d}C}{\alpha} \eqskip
  & = & d\keig{1}{\interv{(V/k)^{1/d}}}{\alpha}\eqskip
  & \geq & \fr{2\alpha d \pi^2 k^{2/d}}{V^{1/d}\left(\pi^2k^{1/d}+2\alpha V^{1/d}\right)},
\end{array}
\]
where $C$ is the unit $d$-dimensional hypercube and we used the lower bound given in Proposition~\ref{prop:firsteiginterv} in the last step.
This will be larger than the right-hand side of~\eqref{lambdakD} provided that
\[
 \fr{4\pi^2}{\left(V\omega_{d}\right)^{2/d}} + \fr{r_{1}(k)}{k^{2/d}} <  \fr{2\alpha d \pi^2 }{V^{1/d}\left(\pi^2k^{1/d}+2\alpha V^{1/d}\right)}.
\]
We may now solve this with respect to $\alpha$ and obtain that if
\[
 \alpha > \frac{2 \pi ^2 }{d \omega ^{2/d}-4}\left(\fr{k}{V}\right)^{1/d} + r_{2}(k),
\]
where $r_{2}(k) = \so\left(k^{1/d}\right)$ as $k$ goes to infinity, then the $k$ equal hypercubes are no longer optimal. In the planar case
and for area $A$ the above reads as
\[
 \alpha > \frac{ \pi ^2 }{\pi -2}\left(\fr{k}{A}\right)^{1/2} + r_{2}(k) \approx 8.64547 \left(\fr{k}{A}\right)^{1/2} + \so(k^{1/2}),
\]
which is comparable to the result in Section~\ref{sec:transition} for the transition between $k$ equal squares and $k-3$ equal squares
and one larger square.

In a similar fashion, it is possible to derive the asymptotic behaviour for the boundary of the region where $k$ equal hypercubes yield
a lower value than a fixed disjoint union of hyperrectangles $\Omega$. Starting from
\begin{equation}\label{lambdakN}
 \keig{k}{\Omega}{\alpha} > \keign{k}{\Omega}= \fr{4\pi^2}{\left(V\omega_{d}\right)^{2/d}} k^{2/d} + r_{3}(k),
\end{equation}
where again the remainder term satisfies $r_{3}(k) = \so\left(k^{2/d}\right), \mbox{ as } k\to +\infty$, and is independent of $\alpha$.
Proceeding as above, but now using the upper bound given in Proposition~\ref{prop:firsteiginterv}
we obtain
\[
\begin{array}{lll}
  \keig{k}{V^{1/d}\unioncube{k}}{\alpha} & \leq & \fr{d\pi^2 k^{2/d}}{2(\pi^2-8)V^{2/d}}
  \Biggr[\pi^2+2\alpha (k^{-1}V)^{1/d} \eqskip
  & & \hspace*{5mm} - \sqrt{ 64\alpha  (k^{-1}V)^{1/d} - \left(\pi^2-2\alpha  (k^{-1}V)^{1/d}\right)^{2}}
  \Biggr],
\end{array}
\]
Comparing this with the right-hand side in~\eqref{lambdakN} and proceeding in the same way as before yields, after some lengthy
calculations, that $k$ equal squares are better than $\Omega$ for
\[
 \alpha < \fr{2}{d \omega_{d}^{2/d}}\left( \pi^2 +\fr{32}{d \omega_{d}^{2/d}-4}\right)\left(\fr{k}{V}\right)^{1/d} + r_{4}(k),
\]
where $r_{4}=\so\left(k^{1/d}\right)$ as $k$ goes to infinity. In the planar case we obtain
\[
 \alpha < \left(\pi +\fr{16}{\pi(\pi-2)}\right)\left(\fr{k}{A}\right)^{1/2} + r_{4}(k) \approx 7.60287 \left(\fr{k}{A}\right)^{1/2} + \so(k^{1/2}).
\]

\appendix

\section{The eigenvalues of the Robin Laplacian on intervals and rectangles\label{sec:interval}}

Here we give sharp bounds for the first and second eigenvalues of the Robin Laplacian on an interval of length $a$, as these are used to build the eigenvalues of rectangles and disjoint unions of rectangles. As these are of independent interest, and to the best of our knowledge many are new, we give sharper estimates than we actually need in many cases. Depending on the particular instance, we may however need our bounds to behave in an appropriate fashion in the different limits of interest, namely, as $a$ and $\alpha$ approach either $0$ or infinity; in such cases, we will present complementary bounds and the corresponding asymptotic expansions.

\subsection{The first eigenvalue on an interval}
The first of these eigenvalues, $\keig{1}{\interv{a}}{\alpha}$, belongs to the interval $(0,\pi^2/a^2)$ and is thus given by the smallest positive
root of
\begin{equation}
\label{eq:1eig-interval}
	\alpha = \sqrt{\lambda} \tan \left(\frac{a\sqrt{\lambda}}{2}\right).
\end{equation}
Expanding the tangent around zero allows us to obtain a formal expression for the expansion of this eigenvalue as $a$ approaches zero as follows
\begin{equation}
\label{eq:eig1expa}
 \keig{1}{\interv{a}}{\alpha} = \fr{2\alpha}{a}-\fr{\alpha^{2}}{3}+\fr{2\alpha^{3}}{45} a - \fr{4\alpha^{4}}{945}a^2 + \fr{2\alpha^{5}}{1475}a^3+\bo(a^4).
\end{equation}
Inserting the above expression in equation~\eqref{eq:1eig-interval}, we see that the argument of the tangent does go to zero as $a$ approaches zero, validating the
expansion.
On the other hand, expanding the tangent around $\pi/2$ yields the corresponding expansion
\[
 \keig{1}{\interv{a}}{\alpha} = \frac{\pi ^2}{a^2} -\frac{4 \pi ^2}{a^3 \alpha }+\frac{12 \pi ^2}{a^4 \alpha ^2}
 -\frac{4 \pi ^2 \left(24-\pi^2\right)}{3 a^5 \alpha ^3} + \bo(\alpha^{-4})
\]
for large $\alpha$.

A first simple remark comes from the fact that, on $(0,\pi/2)$, the tangent is bounded from below by its argument. We thus immediately derive
from~\eqref{eq:1eig-interval} that
\begin{equation}
\label{eq:main-interval-bound}
 \keig{1}{\interv{a}}{\alpha} \leq \fr{2\alpha}{a},
\end{equation}
that is, the first Robin eigenvalue on an interval of length $a$ is smaller that the first term in its expansion as $a$ approaches zero.
Since the corresponding expansion~\eqref{eq:eig1expa} seems to alternate with decreasing terms in absolute value, it is in fact expected that the successive
terms will provide upper and lower bounds for this quantity.  Using further inequalities from the tangent expansion at zero it is also possible to
obtain slightly better albeit more complicated bounds, of which the next using $\tan(x)\geq x + x^3/3$ yields
\[
 \keig{1}{\interv{a}}{\alpha} \leq \frac{2 \sqrt{6 a \alpha +9}-6}{a^2} = \fr{2\alpha}{a}-\fr{\alpha^{2}}{3} + \bo(a), \mbox{ as } a\to0,
\]
with the correct asymptotic behaviour up to the second term.

However, for most of our purposes it will be convenient to obtain bounds with a different form which behave at least in a qualitatively correct
way in more than one asympotic limit. To do this, we shall use a different family of inequalities for the tangent, namely~\cite{best},
\begin{equation}\label{tanbounds}
 \fr{8x}{\pi^2-4x^2}\leq \tan x \leq \fr{\pi^2x}{\pi^2 -4x^2}, \; x\in(0,\fr{\pi}{2}).
\end{equation}
Replacing these in equation~\eqref{eq:1eig-interval} we obtain, after some simplifications,
\[
 \fr{4a\lambda}{\pi^2-a^2\lambda}\leq \alpha \leq \fr{\pi^2 a\lambda}{2(\pi^2-a^2\lambda)}.
\]
Using the fact that we are looking for solutions on the interval $(0,\pi/2)$, we arrive at the following two-sided bounds
\begin{equation}\label{eq:eig1ineq1}
 \fr{2\alpha\pi^2}{a(\pi^2+2\alpha a)}\leq \keig{1}{\interv{a}}{\alpha} \leq \fr{\alpha\pi^2}{a(4+\alpha a)}.
\end{equation}
Both bounds have the first correct term in the corresponding asymptotics when $\alpha$ goes to infinity, and the lower bound also displays the correct
behaviour as $a$ approaches zero. However, this is not the case for the upper bound. In order to obtain a bound that does so, we will use
a test function of the form
\[
 u(x) = 1- c \cos\left(\fr{\pi x}{a}\right),
\]
where $c$ is a constant (possibly depending on $a$ and $\alpha$) to be determined later. Replacing this in the Rayleigh quotient
for $\keig{1}{\interv{a}}{\alpha}$ yields
\[
 \begin{array}{lll}
  \keig{1}{\interv{a}}{\alpha} & \leq & \fr{\dint_{-a/2}^{a/2}\left(\fr{\pi c}{a}\right)^2\sin^{2}\left(\fr{\pi x}{a}\right)\;{\rm d}x
  +2\alpha \left[ 1- c \cos\left(\fr{\pi}{2}\right)\right]^2}{\dint_{-a/2}^{a/2}\left[ 1- c \cos\left(\fr{\pi x}{2}\right)\right]^2\;{\rm d}x}\eqskip
  & = & \fr{\pi \left( \pi^2 c^2+4a\alpha \right)}{a^2\left(2\pi - 8 c + c^2\pi\right)}.
 \end{array}
\]
We now pick the constant $c$ minimising the quotient on the right. This is achieved for
\[
 c= \fr{\pi^2 -2a \alpha-\sqrt{64a \alpha + (\pi^2-2a \alpha)^2}}{4\pi},
\]
which, when replaced back into the above bound, yields
\[
 \keig{1}{\interv{a}}{\alpha} \leq \fr{\pi^2}{a^2}\times \fr{\pi^2 + 2a \alpha - \sqrt{64a\alpha
  +(\pi^2-2a\alpha)^2}}{2(\pi^2-8)}.
\]
It is simple to check that the above bound does satisfy the asymptotic behaviour for both large $\alpha$ and small $a$.

We have thus proved the following
\begin{proposition}\label{prop:firsteiginterv}
 The first eigenvalue of the Robin Laplacian on an interval satisfies
 \[
  \fr{2\alpha\pi^2}{a(\pi^2+2\alpha a)}\leq \keig{1}{\interv{a}}{\alpha} \leq \fr{\pi^2}{a^2}\times \fr{\pi^2 + 2a \alpha - \sqrt{64a\alpha
  +(\pi^2-2a\alpha)^2}}{2(\pi^2-8)}.
 \]
\end{proposition}
The lower bound is accurate up to the first term in the asymptotics for both the small $a$ and large $\alpha$ cases, while the upper
bound is also accurate up to first order in the small $a$ case and to second order in the large $\alpha$ case.

\subsection{The second eigenvalue on an interval}
In a similar way as above, the second eigenvalue $\keig{2}{\interv{a}}{\alpha}$ is obtained as the
smallest solution of the equation
\begin{equation}\label{}
 -\sqrt{\lambda} = \alpha \tan\left(\fr{a \sqrt{\lambda}}{2}\right),
\end{equation}
which is now on the interval $(\pi^2/a^2,4\pi^2/a^2)$. For convenience, we rewrite this equation as
\begin{equation}
\label{eq:2eig-interval}
 \alpha = -\sqrt{\lambda} \cot\left(\fr{a \sqrt{\lambda}}{2}\right),
\end{equation}
and now expand the cotangent around $\pi/2$ to obtain
\[
 \keig{2}{\interv{a}}{\alpha} = \fr{\pi^2}{a^2}+\fr{4\alpha}{a}-\fr{4\alpha^{2}}{\pi^2}+\fr{4(12-\pi^2)\alpha^{3}}{3\pi^4} a
 - \fr{8(10-\pi^2)\alpha^{4}}{\pi^6}a^2+\bo(a^3),
\]
for small $a$.
Again note that if the resulting expression is plugged back into~\eqref{eq:2eig-interval}, the argument of the cotangent approaches $\pi/2$ as $a$ goes
to zero.

A first obvious remark is that $\keig{1}{\interv{a}}{\alpha}$ and $\keig{2}{\interv{a}}{\alpha}$ display a different asymptotic behaviour as $a$
goes to zero; indeed, $\keig{2}{\interv{a}}{\alpha}$ has the same first term as the second Neumann eigenvalue (or first Dirichlet); while the fact that $\keig{1}{\interv{a}}{\alpha} \sim 2\alpha/a$ is what will drive our estimate on $\keig{k}{\unionsquare}{\alpha}$ in Proposition~\ref{prop:boundkequalsquares} below.

For large $\alpha$ we have
\[
 \keig{2}{\interv{a}}{\alpha} =\frac{4 \pi ^2}{a^2}-\frac{16 \pi ^2}{\alpha  a^3}+\frac{48 \pi ^2}{\alpha ^2 a^4}-\frac{128 \pi ^2}{\alpha ^3 a^5} +\frac{320 \pi ^2}{\alpha ^4 a^6} + \bo(\alpha^{-5}).
\]

We will now proceed as in the case of the first eigenvalue to obtain upper and lower bounds which are sharp. We first go back to
equation~\eqref{eq:1eig-interval} and use the inequality
\[
 \tan x \geq \fr{2}{\pi-2x}
\]
valid for $x$ on $(\pi/2,\pi)$ to obtain
\[
 -\sqrt{\lambda} \geq \fr{-2\alpha}{a \sqrt{\lambda}-\pi}.
\]
Since $\sqrt{\lambda}>\pi/a$, we get $a \lambda -\pi \sqrt{\lambda}-2\alpha \leq 0$, yielding the following bounds 
\[
 \fr{\pi}{2a} -\sqrt{\fr{\pi^2}{4a^2}+\fr{2\alpha}{a}} \leq \sqrt{\lambda} \leq \fr{\pi}{2a} +\sqrt{\fr{\pi^2}{4a^2}+\fr{2\alpha}{a}}.
\]
Of these, clearly only the upper bound is of interest, and, in fact, it satisfies
\begin{equation}\label{eig2uppbound1}
\begin{array}{lll}
 \keig{2}{\interv{a}}{\alpha} & \leq & \left(\fr{\pi}{2a} +\sqrt{\fr{\pi^2}{4a^2}+\fr{2\alpha}{a}}\right)^2\eqskip
 & = & \frac{\pi ^2}{a^2}+\frac{4 \alpha }{a} -\frac{4 \alpha ^2}{\pi^2}+ \bo(a),
\end{array}
\end{equation}
as $a$ approaches zero, thus having the same first three terms in the asymptotics as $\keig{2}{\interv{a}}{\alpha}$.

To obtain a sharp lower bound, and also an upper bound which is better that the above for large values of $a\alpha$, we
now use the identity
\[
 \tan x = \fr{ 1-\cos(2x)}{\sin(2x)}
\]
in~\eqref{eq:1eig-interval}. This yields that the second eigenvalue is given by the smallest positive root of the equation
\[
 -\sqrt{\lambda}\sin(a \sqrt{\lambda}) = \alpha\left[ 1-\cos(a \sqrt{\lambda})\right].
\]
Using the inequalities
\[
 \fr{4}{\pi^2} (x-\pi)(x-2\pi) \leq \sin x \leq \fr{1}{\pi^2} (x-\pi)(x-2\pi)
\]
and
\[
 \fr{2}{\pi^2} (x-2\pi)^2 \leq 1-\cos x \leq \fr{2}{\pi^4} (x-2\pi)^2x^2,
\]
valid on $(\pi,2\pi)$, we are led to
\[
-\fr{\sqrt{\lambda}}{\pi} (a\sqrt{\lambda}-\pi)(a\sqrt{\lambda}-2\pi) \leq
 -\sqrt{\lambda}\sin(a \sqrt{\lambda}) = \alpha\left[ 1-\cos(a \sqrt{\lambda})\right] \leq 
 \fr{2\alpha}{\pi^4}(a \sqrt{\lambda}-2\pi)^2a^2\lambda
\]
and
\[
 \fr{2\alpha}{\pi^2}(a \sqrt{\lambda}-2\pi)^2\leq \alpha\left[ 1-\cos(a \sqrt{\lambda})\right] = -\sqrt{\lambda}\sin(a \sqrt{\lambda})
 \leq -\fr{4\sqrt{\lambda}}{\pi^2} (a\sqrt{\lambda}-\pi)(a\sqrt{\lambda}-2\pi).
\]
In the range under consideration, these are, in turn, equivalent to
\[
 2\alpha a^3\lambda + \pi a(\pi^2-4a \alpha)\sqrt{\lambda} - \pi^4 \leq 0
\]
and
\[
 2a \lambda +(a\alpha -2\pi)\sqrt{\lambda}-2\alpha\pi\geq 0,
\]
respectively. The first of these inequalities yields the upper bound
\begin{equation}\label{eig2uppbound2}
 \sqrt{\keig{2}{\interv{a}}{\alpha}}\leq\fr{\pi}{4 a^2\alpha} \left( 4a \alpha -\pi^2 + \sqrt{\pi^4+16a^2\alpha^2} \right),
\end{equation}
while from the second we obtain
\[
 \fr{2\pi -a\alpha + \sqrt{4\pi^2+12a\alpha \pi + \alpha^2 a^2}}{4a} \leq
 \sqrt{\keig{2}{\interv{a}}{\alpha}}.
\]
Comparing the two upper bounds~\eqref{eig2uppbound1} and~\eqref{eig2uppbound2} we see that 
\[
\begin{array}{ll}
 & \fr{\pi}{4 a^2\alpha} \left( 4a \alpha -\pi^2 + \sqrt{\pi^4+16a^2\alpha^2}\right)-
 \left(\fr{\pi}{2a} +\sqrt{\fr{\pi^2}{4a^2}+\fr{2\alpha}{a}}\right)\eqskip
 = & \fr{\pi}{2a} -\fr{\pi^3}{4a^2\alpha}+\fr{\pi}{4a^2\alpha}\sqrt{\pi^4+16a^2\alpha^2}-\fr{1}{2a}\sqrt{\pi^2+8a\alpha}\eqskip
 = & \fr{\pi}{4a^2\alpha}\left( 2b-\pi^2+\sqrt{\pi^4+16a^2\alpha^2}-2a \alpha \sqrt{1+\fr{8a\alpha}{\pi^2}}\right)\eqskip
 = & \fr{\pi}{4ab}\left[ 2b\left(1-\sqrt{1+\fr{8b}{\pi^2}}\right) - \pi^2\left(1-\sqrt{1+\fr{16b^2}{\pi^4}}\right)\right],
\end{array}
\]
where we have written $b=a\alpha$. Simplifying the expression inside the square brackets we see that it vanishes when either $b=0$
or $b=\pi^2/2$, and that it is negative for $b$ on $(0,\pi^2/2)$ and positive for $b$ larger than $\pi^2/2$.

We thus have, for the second eigenvalue,
\begin{proposition}
 The second eigenvalue of the Robin Laplacian on an interval satisfies
 \[
 \fr{\left(2\pi -a\alpha + \sqrt{4\pi^2+12a\alpha \pi + \alpha^2 a^2}\right)^2}{16a^2} \leq \keig{2}{\interv{a}}{\alpha}
 \]
 and
 \[
  \keig{2}{\interv{a}}{\alpha} \leq
  \left\{
  \begin{array}{ll}
    \left(\fr{\pi}{2a} +\sqrt{\fr{\pi^2}{4a^2}+\fr{2\alpha}{a}}\right)^2, & a\alpha\leq \fr{\pi^2}{2}\eqskip
    \fr{\pi^2}{16 a^4\alpha^2} \left( 4a \alpha -\pi^2 + \sqrt{\pi^4+16a^2\alpha^2}\right)^2, & a\alpha\geq \fr{\pi^2}{2}.
    \end{array}
  \right.
 \]
\end{proposition}
All these bounds are accurate up to the first term in the asymptotics as either $a$ becomes small or $\alpha$ large, except for the
upper bound which is accurate up to the third term in the asymptotics as $a$ approaches zero.

\subsection{Bounds for the eigenvalues of rectangles}

The estimates obtained above may now be used to derive bounds for eigenvalues of rectangles. 
The first eigenvalue of a rectangle with side lengths $A^{1/2}a$ and area $A^{1/2}/a$ a particular case of~\eqref{eq:separation-of-variables}
and is given by
\begin{displaymath}
 \keig{1}{\rect{a}{A}}{\alpha} = \keig{1}{\interv{A^{1/2}a}}{\alpha} + \keig{1}{\interv{A^{1/2}/a}}{\alpha}.
\end{displaymath}
It is thus possible to bound this from above and below by means of the bounds from the previous sections, with the same being
possible for the second eigenvalue of rectangles. The expressions do get quinte involved though, and we will concentrate 
on one of the cases which is relevant throughout the paper, namely, the $k^{\rm th}$ eigenvalue of the disjoint union of $k$
equal squares $\unionsquare{k}$ (assumed here to have total area $A$), which coincides with the first eigenvalue of each of the squares. For
a total area $A$, we are thus interested in
\[
 \keig{k}{ \unionsquare{k} }{\alpha} = \keig{1}{\sq{(A/k)^{1/2}}}{\alpha}
 = 2 \keig{1}{\interv{(A/k)^{1/2}}}{\alpha}.
\]
The corresponding bounds obtained directly from Proposition~\ref{prop:firsteiginterv} are as follows.
\begin{proposition}\label{prop:boundkequalsquares}
 The $k^{\rm th}$ eigenvalue of the union of $k$ equal squares with total area $A$ satisfies
\[
\begin{array}{lll}
 \fr{4\alpha\pi^2 k}{A^{1/2}\left(\pi^2 k^{1/2}+2\alpha A^{1/2}\right)} & \leq & \keig{k}{ \unionsquare{k} }{\alpha}\eqskip
 & & 
 \leq \fr{\pi^2k^{1/2}}{(\pi^2-8)A}\times\left[ \pi^2k^{1/2}+2\alpha A^{1/2} \right. \eqskip
 & & \hspace*{5mm}\left.-\sqrt{64\alpha k^{1/2}A^{1/2}+(\pi^2k^{1/2}-2\alpha A^{1/2})^2}\right].
\end{array}
\]
\end{proposition}

\providecommand{\bysame}{\leavevmode\hbox to3em{\hrulefill}\thinspace}
\providecommand{\href}[2]{#2}

\end{document}